\documentclass[11pt,a4paper,lualatex]{amsart}
\usepackage[marginparwidth=0pt,margin=20truemm]{geometry} % margin

\usepackage{mypackage} % よく使うパッケージ. 「C:\w32tex\share\texmf-local\tex\(好きなファイル名)」に置いたmypackage.styを読み込む
\usepackage{mycommand} % 自分で定義したコマンド. 「C:\w32tex\share\texmf-local\tex\(好きなファイル名)」に置いたmycommand.styを読み込む

\DeclareMathOperator{\WRT}{WRT}
\usepackage[pdfencoding=auto,hypertexnames=false]{hyperref}
\hypersetup{colorlinks=false}

\numberwithin{equation}{section} % Setting of equation numbers 
\usepackage{mytheoremeng} % My theorem environments (English) with cleveref package

\begin{document}
% --------------------------------------------------------------------------

\title[A proof of GPPV conjecture]{A proof of a conjecture of Gukov--Pei--Putrov--Vafa}

\author[Y. Murakami]{Yuya Murakami}
\address{Mathematical Inst. Tohoku Univ., 6-3, Aoba, Aramaki, Aoba-Ku, Sendai 980-8578, Japan.}
\email{imaiburngermany@outlook.jp}

\date{\today}

\maketitle

% --------------------------------------------------------------------------

\begin{abstract}
	In the context of $ 3 $-manifolds, determining the asymptotic expansion of the Witten--Reshetikhin--Turaev invariants and constructing the topological field theory that provides their categorification remain important unsolved problems.
	Motivated by solving these problems, Gukov--Pei--Putrov--Vafa refined the Witten--Reshetikhin--Turaev invariants from a physical point of view.
	From a mathematical point of view, we can describe that they introduced new $ q $-series invariants for negative definite plumbed manifolds and conjectured that their radial limits coincide with the Witten--Reshetikhin--Turaev invariants.
	In this paper, we prove their conjecture.
	In our previous work, the author attributed this conjecture to the holomorphy of certain meromorphic functions by developing an asymptotic formula based on the Euler--Maclaurin summation formula. However, it is challenging to prove holomorphy for general plumbed manifolds.	
	In this paper, we address this challenge using induction on a sequence of trees obtained by repeating ``pruning trees,'' which is a special type of the Kirby moves.
\end{abstract}

% --------------------------------------------------------------------------

\tableofcontents

% --------------------------------------------------------------------------

\section{Introduction} \label{sec:intro}

% --------------------------------------------------------------------------

Quantum invariants are important objects in $ 3 $-dimensional topology and relate to mathematical physics, representation theory, and number theory.
Studying its asymptotic behaviour plays a significant role in studying quantum invariants.
In this paper, we study the relationship between two quantum invariants.
The first one is the Witten--Reshetikhin--Turaev (WRT) invariants, which are quantum invariants of $ 3 $-manifolds constructed by Witten~\cite{Witten} and Reshetikhin--Turaev~\cite[Theorem 3.3.2]{Reshetikhin-Turaev} from physical and mathematical viewpoints respectively.
The second one is Gukov--Pei--Putrov--Vafa (GPPV) invariants (also called homological blocks and $ \hat{Z} $-invariants), which are quantum invariants of plumbed $ 3 $-manifolds with negative definite linking matrices.
GPPV invariants are introduced by Gukov--Pei--Putrov--Vafa~\cite{GPPV} and are expected to have two properties:
\begin{enumerate}[label=(\Alph*)]
	\item \label{item:expected_property1} Their radial limits coincide with the WRT invariants;
	\item \label{item:expected_property2} They have good modular transformation properties.
\end{enumerate}
It is expected that we can apply these two properties to prove the Witten's asymptotic expansion conjecture, which is an important conjecture that claims the asymptotic behaviour of the WRT invariants can be written by the Chern--Simons invariants and the Reidemeister torsions.
The properties \cref{item:expected_property1,item:expected_property2} are important in number theory because they imply that the WRT invariants have quantum modularity, which is a number theoretical property introduced by Zagier~\cite{Zagier_quantum}.

Many authors have tried to prove the Witten's asymptotic expansion conjecture in the last two decades.
The pioneering work was done by Lawrence--Zagier~\cite{LZ}.
They proved the conjecture for the Poincar\'{e} homology sphere by constructing a false theta function which has the above two properties:
a relation between its radial limits and the WRT invariants and good modular transformation properties.
Later, Hikami~\cite{H_Bries,H_Seifert} proved the conjecture for Brieskorn homology spheres and Seifert homology spheres, respectively, by constructing more general false theta functions. 
Recently, Gukov--Pei--Putrov--Vafa~\cite{GPPV} introduced the most general false theta functions in our knowledge, which they named GPPV invariants.
Many other authors have studied the Witten's asymptotic expansion conjecture~\cite{Andersen-Hansen,Andersen-Himpel,Andersen-Mistegard,Andersen,Andersen-Petersen,Beasley-Witten,Charles,Chun,Chung_Seifert,Chung_rational,Chung_resurgent,Chung_SU(N),Charles-Marche_I,Charles-Marche_II,FIMT,Freed-Gompf,GMP,H_Lattice,H_Lattice2,Hansen-Takata,Jeffrey,Rozansky1,Rozansky2,Wu}.

As we mentioned above, GPPV invariants are expected to possess the two properties \cref{item:expected_property1,item:expected_property2} discussed above.
The property \cref{item:expected_property1} is proved by
Andersen--Misteg{\aa}rd~\cite{Andersen-Mistegard} and 
Fuji--Iwaki--Murakami--
Terashima~\cite{FIMT} with a result in
Andersen--Misteg{\aa}rd~\cite{Andersen-Mistegard} for Seifert homology spheres independently,
Mori--Murakami~\cite{MM} for non-Seifert homology spheres whose surgery diagrams are the H-graphs, and Murakami~\cite{M_plumbed} for a  class of plumbed manifolds.
We note that for Seifert homology spheres, Fuji--Iwaki--Murakami--Terashima~\cite{FIMT} introduced a $ q $-series they called the WRT function and they proved its radial limits are WRT invariants.
Andersen--Misteg{\aa}rd~\cite{Andersen-Mistegard} proved that the WRT function is identified with the GPPV invariant.
Moreover, they proved a relation between WRT invariants and radial limits of GPPV invariants by using a different method than \cite{FIMT}.
Thus, it can be concluded that \cite{FIMT} and \cite{Andersen-Mistegard} proved independently the first property of GPPV invariants for Seifert homology spheres.
For a more detailed relation between \cite{FIMT} and \cite{Andersen-Mistegard}, see \cite[pp.~715]{Andersen-Mistegard}.

The property \cref{item:expected_property2}, which is good modular transformation property, is proved by Matsusaka--Terashima~\cite{Matsusaka-Terashima} for Seifert homology spheres and Bringmann--Mahlburg--Milas~\cite{BMM_high_depth} for non-Seifert homology spheres whose surgery diagrams are the H-graphs.
Their works are based on the results by Bringmann--Nazaroglu~\cite{Bringmann-Nazaroglu} and Bringmann--Kaszian--Milas--Nazaroglu~\cite{BKMN_False_modular}, which clarified and proved the modular transformation formulas of false theta functions.

In this paper, we prove the property \cref{item:expected_property1} completely.
To state our results, we need the following notation: let $ \Gamma = (V, E, (w_v)_{v \in V}) $ be a plumbing graph, that is, a finite tree with the vertex set $ V $, the edge set $ E $, and integral weights $ w_v \in \Z $ for each vertex $ v \in V $.
Let $ M(\Gamma) $ be the plumed 3-manifold obtained from $ S^3 $ through the surgery along the diagram defined by $ \Gamma $.
Following \cite{GPPV}, we assume that the linking matrix $ W $ of $ \Gamma $ is negative definite.
In these settings, Gukov--Pei--Putrov--Vafa~\cite{GPPV} defined the GPPV invariant $ \widehat{Z}_{b} (q; M(\Gamma)) $ with $ q \in \bbC $ and $ \abs{q} < 1 $ for $ b \in \left( 2\Z^V + \delta \right) /2W(\Z^V) $, where $ \delta \coloneqq (\deg(v))_{v \in V} \in \Z^V $. 
This $ q $-series is a topological invariant of $ M(\Gamma) $ (\cite[pp.~76]{GPPV}).
For a positive integer $ k $, let $ \WRT_k(M(\Gamma)) $ be the WRT invariant of $ M(\Gamma) $ normalised as $ \WRT_k(S^3) = 1 $. 
We also denote $ \zeta_k \coloneqq e^{2\pi\iu/k} $ and $ \bm{e}(z) \coloneqq e^{2\pi\iu z} $ for a complex number $ z $.
Then, we can state Gukov--Pei--Putrov--Vafa~\cite{GPPV}'s conjecture as follows.

\begin{conj}[{\cite[Conjecture 2.1]{GPPV} and \cite[Conjecture 3.1]{GM} for negative definite plumbed manifolds, \cite[Equation (A.28)]{GPPV}}]
	\label{conj:GPPV}
	For each $ b \in ( 2\Z^V + \delta ) /2W(\Z^V) $, the limit $ \lim_{q \to \zeta_k} \widehat{Z}_{b} (q; M(\Gamma)) $ converges and it holds
	\[
	\WRT_k(M(\Gamma))
	= 
	\frac{1}{2(\zeta_{2k} - \zeta_{2k}^{-1}) \sqrt{\abs{\det W}}}
	\sum_{\substack{
			a \in \Z^V/W(\Z^V), \\
			b \in ( 2\Z^V + \delta ) /2W(\Z^V)
	}}
	\bm{e} \left( -k {}^t\!a W^{-1} a - {}^t\!a W^{-1} b \right)
	\lim_{q \to \zeta_k}
	\widehat{Z}_{b} (q; M(\Gamma)).
	\]
\end{conj}

In this paper, we prove this conjecture for general negative definite pumbed manifolds without convergency of $ \lim_{q \to \zeta_k} \widehat{Z}_{b} (q; M(\Gamma)) $.

\begin{thm} \label{thm:main}
	\begin{align}
		&\phant
		\mathrm{WRT}_k(M(\Gamma))
		\\
		&=
		\frac{1}{2(\zeta_{2k} - \zeta_{2k}^{-1}) \sqrt{\abs{\det W}}}
		\lim_{t \to +0}
		\sum_{\substack{
				a \in \Z^V/W(\Z^V), \\
				b \in ( 2\Z^V + \delta ) /2W(\Z^V)
		}}
		\bm{e} \left( -k {}^t\!a W^{-1} a - {}^t\!a W^{-1} b \right)
		\restrict{\widehat{Z}_{b} (q; M(\Gamma))}{q = \zeta_k e^{-t}}.
	\end{align}
\end{thm}

The idea of our proof is the following:
\begin{enumerate}[leftmargin=*, labelsep=1em, label=\textbf{Step \Alph*}]
	\item \label{item:proof_stepA}
	Murakami~\cite{M_plumbed} proved an asymptotic formula of infinite series of sufficiently general forms by using the Euler--Maclaurin summation formula.
	This formula implies a useful property that asymptotic expansions of different infinite series relate.
	Using this property, we develop a new formula that relates Laurent expansions of meromorphic functions and asymptotic expansions of infinite series that Cauchy principal values appear in its summands.
	
	\item \label{item:proof_stepB}
	By the formula in \cref{item:proof_stepA}, we construct meromorphic functions whose Laurent expansion relates to asymptotic expansions of the GPPV invariants.
	If we prove that these meromorphic functions are holomorphic at the origin and their constant terms coincide with the WRT invariants, then we obtain that radial limits of the GPPV invariants converge and coincide with the WRT invariants, and we complete a proof.
	
	\item \label{item:proof_stepC}
	We prove the required properties of meromorphic functions in \cref{item:proof_stepB} by induction on a sequence of trees obtained by repeating ``pruning trees.''
\end{enumerate}

In \cref{item:proof_stepA}, Murakami~\cite{M_plumbed} applied his asymptotic formula to the GPPV invariants under the assumption $ \det W = \pm 1 $.
To apply general $ W $, we developed a new asymptotic formula (\cref{prop:asymp_F_v}), which is the aim in \cref{item:proof_stepA}.
%This formula clarifies the rationale of that Gukov--Pei--Putrov--Vafa~\cite{GPPV} used Cauchy principal values in the definition of the GPPV invariants.
In \cref{item:proof_stepB}, the idea of constructing meromorphic functions is due to Murakami~\cite{M_plumbed}.
However, we provide simpler and more pertinent construction.
\cref{item:proof_stepC} is the most difficult part of our proof.
This part corresponds to the vanishing of the asymptotic coefficients of the GPPV invariant in negative degrees, which was also the most difficult part in proofs in \cite{MM,M_plumbed}.
In \cite{MM}, this is called ``vanishing of weighted Gauss sums,'' and was proved by a complex and technical argument in elementary number theory.
In \cite{M_plumbed}, it is proved more sophisticatedly by using a simple case of ``vanishing of weighted Gauss sums.''
We prove this part by induction on the operation  ``pruning trees.''
This operation involves ``pruning'' vertices of degree $ 1 $ on plumbed graphs. 
This has already been employed in \cite[Section 21]{Eisenbud-Neumann} for diagonalizing the linking matrix of plumbed graphs.
This operation also relates to the operation of constructing a periodic map by grouping vertices of degree $ 1 $ that appeared in previous works~\cite{LZ,H_Bries,H_Seifert,MM,M_plumbed}.
%We can consider this operation as a move of plumbing graphs.

This paper will be organised as follows. 
In \cref{sec:preliminaries}, we prepare some basic notations and facts which we use throughout this paper.
In \cref{sec:HB_asymp}, we develop a new asymptotic expansion formula which we mentioned in \cref{item:proof_stepA}.
Moreover, we construct meromorphic functions in \cref{item:proof_stepB}.
Finally, we carry out \cref{item:proof_stepC} in \cref{sec:pruned}.

% --------------------------------------------------------------------------

\section*{Acknowledgement} \label{sec:acknowledgement}

% --------------------------------------------------------------------------

The author would like to show the greatest appreciation to Takuya Yamauchi for giving much advice. 
The author would like to thank Yuji Terashima, Kazuhiro Hikami, and Toshiki Matsusaka for giving many comments. 
The author is deeply grateful to Akihito Mori for a lot of discussions. 
%I thank the referees for their helpful suggestions and comments
The author thank the referees for their helpful suggestions and comments which improved the presentation of our paper.
The author is supported by JSPS KAKENHI Grant Number JP 20J20308.

% --------------------------------------------------------------------------

\section{Preliminaries} \label{sec:preliminaries}

% --------------------------------------------------------------------------

In this section, we provide some notations and basic facts, which we use throughout this paper.

% --------------------------------------------------------------------------

\subsection{Notations for graphs} \label{subsec:graph}

% --------------------------------------------------------------------------

In this subsection, we prepare settings for graphs.
As in \cref{sec:intro}, let $ \Gamma = (V, E, (w_v)_{v \in V}) $ be a plumbing graph and $ W $ be its linking matrix.
Here, we consider the edge set $ E $ as the subset of $ \{ \{ v, v' \} \mid v, v' \in V \} $.
%We identify $ M_{\abs{V}}(\Z) $ and $ \End(\Z^V) $ and consider $ W $ as the element of $ \End(\Z^V) $.
Throughout this paper, we assume that $ W $ is negative definite.
We remark that $ w_v \in \Z_{<0} $ for any vertex $ v \in V $.

For two plumbing graphs $ \Gamma $ and $ \Gamma' $, Neumann (\cite[Proposition 2.2]{Neumann_Lecture}, \cite[Theorem 3.1]{Neumann_work}) proved that two $ 3 $-manifolds $ M(\Gamma) $ and $ M(\Gamma') $ are homeomorphic if and only if $ \Gamma $ and $ \Gamma' $ are related by Neumann moves shown in \cref{fig:Neumann}.
Thus, we can assume $ w_v \le -2 $ for a vertex $ i $ with $ \deg(i) = 1 $. 
Although the present works~\cite{MM,M_plumbed} assumed it, we do not need to assume it.

\begin{figure}[htp]
	\centering
	\begin{tikzpicture}
		%Move I
		\draw[fill]
		%隣接3頂点
		(-1.5,0) node[above=0.1cm]{$w \pm 1$} circle(0.5ex)--
		(0,0) node[above=0.1cm]{$\pm 1$} circle(0.5ex)--
		(1.5,0) node[above=0.1cm]{$w' \pm 1$} circle(0.5ex)
		%左頂点から伸びる辺
		(-2.5,0.5) node[above]{}--(-1.5,0) node[above]{}
		(-2.3,0) node[rotate=270]{$\cdots$}
		(-2.5,-0.5) node[above]{}--(-1.5,-0) node[above]{}
		%右頂点から伸びる辺
		(1.5,0) node[above]{}--(2.5,0.5) node[above]{}
		(2.3,0) node[rotate=270]{$\ldotp\ldotp\ldotp\ldotp$}
		(1.5,0) node[above]{}--(2.5,-0.5) node[above]{}
		%矢印
		(0,-1) node[rotate=270]{$\longleftrightarrow$}
		%隣接2頂点
		(-1,-2) node[above=0.1cm]{$ w $} circle(0.5ex)--
		(1,-2) node[above=0.1cm]{$ w' $} circle(0.5ex)
		%左頂点から伸びる辺
		(-2,-1.5) node[above]{}--(-1,-2) node[above]{}
		(-1.8,-2) node[rotate=270]{$\cdots$}
		(-2,-2.5) node[above]{}--(-1,-2) node[above]{}
		%右頂点から伸びる辺
		(1,-2) node[above]{}--(2,-1.5) node[above]{}
		(1.8,-2) node[rotate=270]{$\ldotp\ldotp\ldotp\ldotp$}
		(1,-2) node[above]{}--(2,-2.5) node[above]{};
		%Move II
		\draw[fill]
		%隣接2頂点
		(4.7,0) node[above=0.1cm]{$w \pm 1$} circle(0.5ex)--
		(5.7,0) node[above=0.1cm]{$\pm 1$} circle(0.5ex)
		%左頂点から伸びる辺
		(3.9,0.5) node[above]{}--(4.7,0) node[above]{}
		(4.1,0) node[rotate=270]{$\cdots$}
		(3.9,-0.5) node[above]{}--(4.7,0) node[above]{}
		%矢印
		(5,-1) node[rotate=270]{$\longleftrightarrow$}
		%1頂点
		(5,-2) node[above=0.1cm]{$w$} circle(0.5ex)
		%左頂点から伸びる辺
		(4,-1.5) node[above]{}--(5,-2) node[above]{}
		(4.2,-2) node[rotate=270]{$\cdots$}
		(4,-2.5) node[above]{}--(5,-2) node[above]{};
		%Move III
		\draw[fill]
		%隣接2頂点
		(8.5,0) node[above=0.1cm]{$w$} circle(0.5ex)--
		(10,0) node[above=0.1cm]{$0$} circle(0.5ex)--
		(11.5,0) node[above=0.1cm]{$w'$} circle(0.5ex)
		%左頂点から伸びる辺
		(7.5,0.5) node[above]{}--(8.5,0) node[above]{}
		(7.7,0) node[rotate=270]{$\cdots$}
		(7.5,-0.5) node[above]{}--(8.5,0) node[above]{}
		%右頂点から伸びる辺
		(11.5,0) node[above]{}--(12.5,0.5) node[above]{}
		(12.3,0) node[rotate=270]{$\ldotp\ldotp\ldotp\ldotp$}
		(11.5,0) node[above]{}--(12.5,-0.5) node[above]{}
		%矢印
		(10,-1) node[rotate=270]{$\longleftrightarrow$}
		%1頂点
		(10,-2) node[above=0.2cm]{$w + w'$} circle(0.5ex)
		%左頂点から伸びる辺
		(9,-1.5) node[above]{}--(10,-2) node[above]{}
		(9.2,-2) node[rotate=270]{$\cdots$}
		(9,-2.5) node[above]{}--(10,-2) node[above]{}
		%右頂点から伸びる辺
		(10,-2) node[above]{}--(11,-1.5) node[above]{}
		(10.8,-2) node[rotate=270]{$\ldotp\ldotp\ldotp\ldotp$}
		(10,-2) node[above]{}--(11,-2.5) node[above]{};
	\end{tikzpicture}
	\caption{Neumann moves}
	\label{fig:Neumann}
\end{figure}
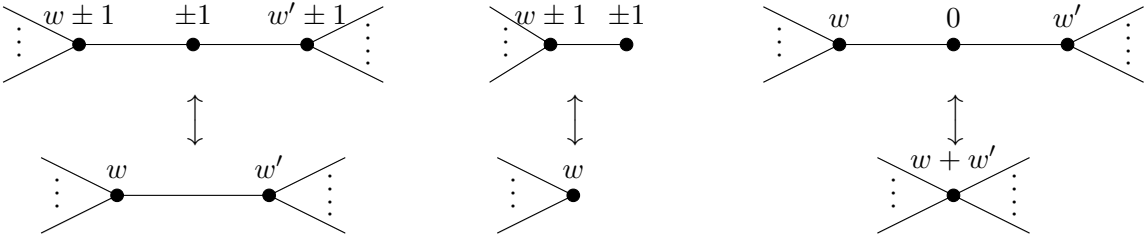

% --------------------------------------------------------------------------

\subsection{The Witten--Reshetikhin--Turaev invariants} \label{subsec:WRT}

% --------------------------------------------------------------------------

For a vertex $ v \in V $, let $ \deg(v) \coloneqq \# \left\{ v' \in V \relmiddle{|} {v, v'} \in E \right\} $ be its degree 
and let $ F_v (z_v) \coloneqq \left( z_v - z_v^{-1} \right)^{2 - \deg(v)} $ be a rational function. 
Gukov--Pei--Putrov--Vafa~\cite{GPPV} expressed the WRT invariants $ \WRT_k(M(\Gamma)) $ of the plumbed homology sphere $ M(\Gamma) $ as follows.

\begin{prop}[{Gukov--Pei--Putrov--Vafa~\cite[Equation A.12]{GPPV}}]
	\label{prop:WRT_rep}
	\begin{align}
		\WRT_k(M(\Gamma))
		=
		\frac{\zeta_8^{\abs{V}} \zeta_{4k}^{-\sum_{v \in V} (w_{v} + 3)}}
		{2 \sqrt{2k}^{\abs{V}} \left( \zeta_{2k} - \zeta_{2k}^{-1} \right)}
		\sum_{\mu \in (\Z \smallsetminus k\Z)^V/2k\Z^V}
		\bm{e} \left( \frac{1}{4k} {}^t\!\mu W \mu \right)
		\prod_{v \in V} F_v \left( \zeta_{2k}^{\mu_v} \right).
	\end{align}
\end{prop}

% --------------------------------------------------------------------------

\subsection{The GPPV invariants} \label{subsec:HB_false_main}

% --------------------------------------------------------------------------

For a negative definite plumbed manifold $ M(\Gamma) $, the GPPV invariant $ \widehat{Z}_{\Gamma} (q) $ is defined as follows.

\begin{dfn}[{\cite[Equation (A.29)]{GPPV}}]
	The GPPV invariant of a negative definite plumbed manifold $ M(\Gamma) $ is defined as
	\[
	\widehat{Z}_{b} (q; M(\Gamma))
	\coloneqq 
	q^{-\sum_{v \in V} (w_{v} + 3)/4} \,
	\mathrm{v.p.} \int_{\abs{z_{v}}=1, v \in V} \Theta_{Q, b}(q; z)
	\prod_{v \in V} F_v (z_{v}) \frac{dz_{v}}{2\pi\sqrt{-1}z_{v}}
	\]
	for $ b \in \Z^V/2W(\Z^V) $, where $ \mathrm{v.p.} $ is the Cauchy principal value defined as
	\[
	\mathrm{v.p.}
	\coloneqq 
	\frac{1}{2}
	\lim_{\veps \to +0} \left( \int_{\abs{z} = 1+\veps} + \int_{\abs{z} = 1-\veps} \right),
	\]
	$ Q \colon \R^V \to \R $ is a positive definite quadratic form defined as $ Q(x) \coloneqq -{}^t\!x W^{-1} x $,
	\[
	\Theta_{Q, b} (q; z) \coloneqq 
	\sum_{l \in 2W(\Z^V) + b} q^{Q(l)/4} \prod_{v \in V} z_v^{l_v}
	\]
	is the theta function defined for complex numbers $ q $ and $ z $ whose absolute values are smaller than $ 1 $,
	and $ F_v (z_v) $ is the rational funtion defined in \cref{subsec:WRT}. 
\end{dfn}

%Here we remark that our definition of the Cauchy principal value is half of it in \cite{GPPV} (see \cite[pp. 55]{GPPV}).
%However, our definition of the GPPV invariant $ \widehat{Z}_{\Gamma} (q) $ is the same as \cite{GPPV} since we multiply it by $ 2^{-\abs{V}} $.

Let $ \delta \coloneqq (\deg(v))_{v \in V} \in \Z^V $.
We have $ \widehat{Z}_{b} (q; M(\Gamma)) = 0 $ for $ b \in \Z^V/2W(\Z^V) \smallsetminus ( 2\Z^V + \delta ) /2W(\Z^V) $ by \cref{lem:coeff_of_F} \cref{item:lem:coeff_of_F:expression} below.

%We need more notation.
%Let
%\begin{align}
%	\Theta_{Q, k} (q; z) \coloneqq
%	\sum_{\alpha \in \Z^V / W(\Z^V)} \bm{e} \left( kQ(\alpha) \right)
%	\sum_{l \in 2\Z^V + \delta} \bm{e} \left( -{}^t\!\alpha W^{-1} l \right)
%	q^{Q(l)/4} \prod_{v \in V} z_v^{l_v}
%\end{align}
%be the theta function defined for complex numbers $ q $ and $ z $ whose absolute values are smaller than $ 1 $.
%
%\begin{rem} \label{rem:GPPV_conj_modify}
%	Under the above notations, we can write
%	\begin{align}
	%		2^{-\abs{V}}
	%		\mathrm{v.p.} \int_{\abs{z_{v}}=1, v \in V} 
	%		\Theta_{Q, k} (q; z)
	%		&=
	%		\sum_{\substack{
			%				\alpha \in \Z^V/W(\Z^V), \\
			%				\beta \in ( 2\Z^V + \delta ) /2W(\Z^V)
			%		}}
	%		\bm{e} \left( -k {}^t\!\alpha W^{-1} \alpha - {}^t\!\alpha W^{-1} \beta \right)
	%		\widehat{Z}_{\Gamma, \beta} (q).
	%	\end{align}
%	Thus, \cref{conj:GPPV} is equivalent to the equality
%	\begin{align}
	%		&\lim_{q \to \zeta_k} 
	%		2^{-\abs{V}}
	%		\mathrm{v.p.} \int_{\abs{z_{v}}=1, v \in V} \Theta_{Q, k}(q; z)
	%		\prod_{v \in V} F_v (z_{v}) \frac{dz_{v}}{2\pi\sqrt{-1}z_{v}}
	%		\\
	%		= \,
	%		&\frac{ \zeta_8^{\abs{V}} \sqrt{\abs{\det W}} }{\sqrt{2k}^{\abs{V}}}
	%		\sum_{\mu \in (\Z \smallsetminus k\Z)^V/2k\Z^V}
	%		\bm{e} \left( \frac{1}{4k} {}^t\!\mu W \mu \right)
	%		\prod_{v \in V} F_v \left( \zeta_{2k}^{\mu_v} \right)
	%	\end{align}
%	by \cref{prop:WRT_rep}. 
%\end{rem}

% --------------------------------------------------------------------------

\subsection{Coefficients of $ F_v (z_v) $} \label{subsec:coeff_of_F}

% --------------------------------------------------------------------------

Recall that we denote $ F_v (z_v) \coloneqq \left( z_v - z_v^{-1} \right)^{2 - \deg(v)} $ for $ v \in V $ in the beginning of \cref{subsec:WRT}.
For $ l = (l_v)_{v \in V} \in \Z^V $, define
\[
F_l \coloneqq \prod_{v \in V} F_{v, l_v}, \quad
F_{v, l_v} \coloneqq \mathrm{v.p.} \int_{\abs{z_v} = 1} F_v (z_v) \frac{z_v^{l_v} dz_v}{2\pi\iu z_v}.
\]
Then, we have
\begin{equation} \label{eq:GPPV:expression}
	\widehat{Z}_{b} (q; M(\Gamma))
	=
	q^{-\sum_{v \in V} (w_{v} + 3)/4}
	\sum_{l \in 2W(\Z^V) + b} F_l q^{Q(l)/4}
\end{equation}
for $ b \in \Z^V/2W(\Z^V) $.

We need the following lemma to consider asymptotic expansions of GPPV invariants and to establish \cref{item:proof_stepA,item:proof_stepB} in \cref{sec:HB_asymp}.

\begin{lem} \label{lem:coeff_of_F}
	Let $ l = (l_v)_{v \in V} \in \Z^V $ and $ v \in V $.
	\begin{enumerate}
		\item \label{item:lem:coeff_of_F:expression}
		\textup{(\cite[p. 743]{Andersen-Mistegard})} We have
		\begin{equation}
			F_{v, l_v} =
			\begin{dcases}
				-l_v, & \text{ if } \deg v = 1, \, l_v \in \{ \pm 1 \}, \\
				1, & \text{ if } \deg v = 2, \, l_v =0, \\
				\frac{\sgn(l_v)^{\deg v}}{2} \binom{m + \deg v - 3}{\deg v - 3}, & \text{ if } \deg v \ge 3, \, \pm l_v = \deg v -2 + 2m \text{ for some } m \in \Z_{\ge 0}, \\
				0, & \text{ otherwise}.
			\end{dcases}
		\end{equation}
		
		\item \label{item:lem:coeff_of_F:symmetry}
		We have $ F_{v, l_v} = (-1)^{\deg v} F_{v, l_v} $ and $ F_{-l} = F_l $.
		
		\item \label{item:lem:coeff_of_F:expansion}
		We have
		\[
		\sum_{l_v \in \deg v + 2 \Z_{\ge -1}} F_{v, l_v} z_v^{l_v}
		=
		(-1)^{\deg v} F_v(z_v) \cdot
		\begin{dcases}
			1, \text{ if } \deg v \le 2, \\
			\frac{1}{2}, \text{ if } \deg v \ge 3
		\end{dcases}
		\]
		and
		\[
		\prod_{v \in V} F_v (z_v)
		=
		2^{\abs{V_{\ge 3}}}
		\sum_{l = (l_v)_{v \in V} \in 2 \Z_{\ge -1}^{V} + \delta} F_{l} \prod_{v \in V} z_v^{l_v},
		\]
		where $ V_{\ge 3} := \{ v \in V \mid \deg v \ge 3 \} $.
	\end{enumerate}
\end{lem}

\begin{proof}
	\cref{item:lem:coeff_of_F:expression} is proved in \cite[p. 743]{Andersen-Mistegard}.
	
	\cref{item:lem:coeff_of_F:symmetry}
	The first equality follows from the symmetry $ F_v (z_v^{-1}) = (-1)^{\deg v} F_v (z_v) $.
	The second equality follows from the first equality and the handshaking lemma.
	
	\cref{item:lem:coeff_of_F:expansion}
	The first equality follows from \cref{item:lem:coeff_of_F:expression} and the binomial theorem.
	The second equality follows from the first equality and the shakehand lemma.
\end{proof}

% --------------------------------------------------------------------------

\subsection{Reciprocity of Gauss sums} \label{subsec:reciprocity}

% --------------------------------------------------------------------------

To calculate the WRT invariants and the GPPV invariants, we need the following formula: ``reciprocity of Gauss sums.''

\begin{prop}[{\cite[Theorem 1]{DT}}] \label{prop:reciprocity}
	Let $ L $ be a lattice of finite rank $ n $ equipped with a non-degenerated symmetric $ \Z $-valued bilinear form $ \sprod{\cdot, \cdot} $.
	We write
	\[
	L' \coloneqq \{ y \in L \otimes \R \mid \sprod{x, y} \in \Z \text{ for all } x \in L \} 
	\]
	for the dual lattice.
	Let $ 0 < k \in \abs{L'/L} \Z, u \in \frac{1}{k} L $, 
	and $ h \colon L \otimes \R \to L \otimes \R $ be a self-adjoint automorphism such that $ h(L') \subset L' $ and $ \frac{k}{2} \sprod{y, h(y)} \in \Z $ for all $ y \in L' $.
	Let $ \sigma $ be the signature of the quadratic form $ \sprod{x, h(y)} $.
	Then it holds
	\begin{align}
		&\sum_{x \in L/kL} \bm{e} \left( \frac{1}{2k} \sprod{x, h(x)} + \sprod{x, u} \right)
		= \,
		&\frac{\bm{e}(\sigma/8) k^{n/2}}{\sqrt{\abs{L'/L} \abs{\det h}}}
		\sum_{y \in L'/h(L')} \bm{e} \left( -\frac{k}{2} \sprod{y + u, h^{-1}(y + u)} \right).
	\end{align}
\end{prop}

% --------------------------------------------------------------------------

\section{Asymptotic expansions of the GPPV invariants} \label{sec:HB_asymp}

% --------------------------------------------------------------------------

In this section, we consider the asymptotic expansion as $ t \to +0 $ of the function obtained by evaluating $ q = \zeta_k e^{-t} $ for
\[
\sum_{\substack{
		a \in \Z^V/W(\Z^V), \\
		b \in ( 2\Z^V + \delta ) /2W(\Z^V)
}}
\bm{e} \left( -k {}^t\!a W^{-1} a - {}^t\!a W^{-1} b \right)
\widehat{Z}_{b} (q; M(\Gamma)),
\]
which appeared in \cref{thm:main}.
To do this, we develop a new formula (\cref{prop:asymp_F_v}) which relates Laurent expansions of meromorphic functions and asymptotic expansions of infinite series that Cauchy principal values appear in its summands.

% --------------------------------------------------------------------------

\subsection{An asymptotic formula by \cite{M_plumbed}} \label{subsec:asymp_previous}

% --------------------------------------------------------------------------

To begin with, we prepare the notation for asymptotic expansion by Poincar\'{e}. 

\begin{dfn}[Poincar\'{e}]
	Let $ L $ be a positive integer, $ \varphi \colon \R_{>0} \to \bbC $ be a map, $ t $ be a variable of $ \R_{>0} $, and $ (a_m)_{m = -L}^{\infty} $ be a family of complex numbers.
	Then, we write
	\[
	\varphi(t) \sim \sum_{m \ge -L} a_m t^{m} 
	\quad \text{as } t \to +0
	\]
	if for any positive number $ M $ there exists a positive number $ K_M $ and $ \varepsilon $ such that
	\[
	\abs{ \varphi(t) - \sum_{-L \le m \le M} a_m t^{m} }
	\le K_M t^{M}
	\]
	for any $ t \in (0, \varepsilon) $.
	In this case, we call the infinite series $ \sum_{m \ge -L} a_m t^{m} $ as the \textbf{asymptotic expansion} of  $ \varphi(t) $ as $ t \to +0 $. 
\end{dfn}

%Furthermore, we introduce the following new notation for asymptotic expansion. 
%
%\begin{dfn}
%	Let $ L $ and $ N $ a positive number and $ f, g \colon \R_{>0}^N \to \bbC $ be maps which admit asymptotic expansions 
%	$ \sum_{n_1, \dots, n_N \ge -L} a_n t_1^{n_1} \cdots t_N^{n_N}$ and $ \sum_{n_1, \dots, n_N \ge -L} a_n t_1^{n_1} \cdots t_N^{n_N} $
%	as $ t = (t_1, \dots, t_N) \to (+0, \dots, +0) $ respectively.
%	Then, we write $ f(t) \succsim g(t) \text{ as } t \to (+0, \dots, +0) $
%	if $ a_0 = b_0 $ and for each $ n \in \Z_{\ge -L}^N $, if $ a_n = 0 $, then $ b_n = 0 $.
%	In this case, we call \textbf{$ f(t) $ dominates the asymptotic expansion of $ g(t) $}.
%\end{dfn}
%
%\begin{rem} \label{rem:dominate}
%	\begin{enumerate}
	%		\item \label{item:rem:dominate:1}
	%		If $ f(t) \succsim g(t) \text{ as } t \to (+0, \dots, +0) $ and $ f(t) $ has the limit value $ f(0) $ as $ t \to (+0, \dots, +0) $, then $ g(t) $ has also the limit value $ g(0) $ as $ t \to (+0, \dots, +0) $ and  it holds $ f(0) = g(0) $. 
	%		\item \label{item:rem:dominate:2}
	%		For maps $ f_1(t), \dots, f_m(t) $, $ g_1(t), \dots, g_m(t) $ and complex numbers $ a_1, \dots, a_m $,
	%		if $ f_i(t) \succsim g_i(t) \text{ as } t \to (+0, \dots, +0) $ for each $ 1\le i \le m $, then 
	%		$ a_1 f_1(t) + \cdots a_m f_m(t) \succsim a_1 g_1(t) + \cdots a_m g_m(t) \text{ as } t \to (+0, \dots, +0) $.
	%	\end{enumerate}
%\end{rem}
%
%The notation $ f(t) \succsim g(t) $ is useful to describe the following asymptotic formulas of \cite{M_plumbed}, which are our main tool to study asymptotic expansions.

We also need the following terminology.

\begin{dfn}
	Let $ N $ be a positive integer and $ D \subset \R^N $ be an unbounded open domain.
	A $ C^\infty $ function $ f \colon D \to \bbC $ is called of \textbf{rapid decay} as $ x_1, \dots, x_N \to \infty $ 
	if $ x_1^{m_1} \cdots x_N^{m_N} f^{(n)} (x) $ is bounded as $ x_1, \dots, x_N \to \infty $ for any $ m, n \in \Z_{\ge 0}^N $.
\end{dfn}

The following proposition is our starting point to study asymptotic expansions.

\begin{prop}[{\cite[Proposition 5.4]{M_plumbed}}] \label{prop:Euler--Maclaurin_poly}
	Let $ N $ and $ N' $ be non-negative integers, $ f \colon \R^{N+N'} \to \bbC $ be a $ C^\infty $ function of rapid decay as $ x_1, \dots, x_N \to \infty $, and $ P(x) = \sum_{m \in \Z_{\ge 0}^N} p_m x_1^{m_1} \cdots x_N^{m_N} $ be a polynomial.
	Fix $ \alpha, \lambda \in \R^N $ and $ \alpha' \in \R^{N'} $.
	Then, for a variable $ t \in \R_{>0} $, an asymptotic expansion
	\begin{align}
		&\sum_{n \in \Z_{\ge 0}^N} P(\lambda + n) f(t(\alpha + \lambda + n), t\alpha') \\
		\sim \,
		&\sum_{n \in \Z^N \times \Z_{\ge 0}^{N'}} t^{n_1 + \cdots + n_{N+N'}} f^{(n)}(0) 
		\frac{{\alpha_1'}^{n_{N+1}} \cdots {\alpha_{N'}'}^{n_{N+N'}}}{n_{N+1}! \cdots n_{N+N'}!} 
		\sum_{m \in \Z_{\ge 0}^N} p_m \mathbb{B}_{m, n}(\alpha, \lambda)
		\quad \text{as } t \to +0 
	\end{align}
	holds.
	Here we define
	\[
	g^{(-1)}(x) = \frac{d^{-1}}{d x^{-1}} g(x) \coloneqq -\int_x^\infty g(x') dx', \quad
	f^{(n)}(x) \coloneqq \frac{\partial^{n_1 + \cdots + n_N} f}{\partial x_1^{n_1} \cdots \partial x_r^{n_N}} (x).
	\]
	and
	\begin{align}
		\mathbb{B}_{m, n}(\alpha, \lambda)
		&\coloneqq
		\prod_{1 \le i \le N} \mathbb{B}_{m_i, n_i} (\alpha_i, \lambda_i), \\
		\mathbb{B}_{m_i, n_i} (\alpha_i, \lambda_i)
		&\coloneqq
		\begin{dcases}
			\sum_{0 \le l \le m_i + n_i + 1} b_{m_i, n_i, l} B_{m_i + n_i + 1 - l}(\lambda_i) \alpha_i^l & \text{ if } n_i \ge 0, \\
			\frac{m_i!}{(m_i + n_i + 1)!} (-\alpha_i)^{m_i + n_i + 1} & \text{ if } -m_i - 1 \le n_i \le -1, \\
			0 & \text{ if } n_i \le -m_i - 2,
		\end{dcases}
		\\
		b_{m_i, n_i, l} 
		&\coloneqq 
		\frac{m_i!}{(m_i + n_i + 1- l)!} \sum_{0 \le k \le l} \pmat{m_i + n_i - k \\ n_i} \frac{(-1)^k}{k! (l-k)!},
	\end{align}
	where $ B_i(x) $ is the $ i $-th Bernoulli polynomial.
\end{prop}

\cref{prop:Euler--Maclaurin_poly} generalizes \cite[Equation (44)]{Zagier_asymptotic}, \cite[Equation (2.8)]{BKM} and \cite[Lemma 2.2]{BMM_high_depth}.
\cref{prop:Euler--Maclaurin_poly} was proved by using the Euler--Maclaurin summation formula.

We transform the above formula in \cref{prop:Euler--Maclaurin_poly} into a simpler and more useful form.
For this purpose, we introduce the following notation.

\begin{dfn} \label{dfn:Hadamard}
	Let $ N $ be a positive integer.
	For a formal Laurent series $ \varphi(t_1, \dots, t_N) = \sum_{m \in \Z^N} B_m t_1^{m_1} \cdots t_N^{m_N} \in \bbC(t_1, \dots, t_N) $
	and a $ C^\infty $ function $ f \colon \R^{N} \to \bbC $ such that $ f^{(m)} (0) $ converges for any $ m \in \Z^N $, define their \textbf{Hadamard product} $ \varphi \odot f (t) \in \bbC((t)) $ as
	\[
	\varphi \odot f (t)
	\coloneqq
	\sum_{m \in \Z^N} B_m f^{(m)}(0) t^{m_1 + \cdots + m_N}.
	\]
\end{dfn}

\begin{rem}
	Usually, for two formal Laurent series
	\[
	\varphi(t_1, \dots, t_N) = \sum_{m \in \Z^N} B_m t_1^{m_1} \cdots t_N^{m_N}, \quad
	\psi(t_1, \dots, t_N) = \sum_{m \in \Z^N} C_m t_1^{m_1} \cdots t_N^{m_N}
	\in \bbC(t_1, \dots, t_N),
	\]
	the formal Laurent series
	\[
	\varphi \odot \psi (t_1, \dots, t_N)
	\coloneqq
	\sum_{m \in \Z^N} B_m C_m t_1^{m_1} \cdots t_N^{m_N}
	\in \bbC(t_1, \dots, t_N)
	\]
	is referred to as their Hadamard product. 
	However, for our purposes, we will only use $ \varphi \odot f (t) $ defined above \cref{dfn:Hadamard}.
	Thus, to simplify notation, we will adopt the above definition.
\end{rem}

We formulate the new asymptotic formula as follows.

\begin{prop} \label{prop:asymp_lim}
	Let $ N $ be a positive integer, $ \lambda \in \Z^N $ and $ F \colon \Z_{\ge 0}^N + \lambda \to \bbC $ be an element of
	\[
	\delta(y \in \Z_{\ge 0}^N + \lambda) \cdot
	\bbC[ y_i, \delta(y_i \in a + k \Z) \mid 1 \le i \le N, a, k \in \Z ],
	\]
	where $ \delta $ be the Kronecker delta function and we stipulate that $ a + k\Z = \{ a \} $ when $ k=0 $.
	Let 
	\[
	\varphi_{F, u} (t_1, \dots, t_N)
	\coloneqq
	\sum_{l \in \Z_{\ge 0}^N + \lambda} F(l) 
	e^{t_1 (l_1 + u_1) + \cdots + t_N (l_N + u_N)}
	\in \bbC ((t_1, \dots, t_N))[u_1, \dots, u_N].
	\]
	Then, for any $ \alpha \in \R^N $ and any $ C^\infty $ function $ f \colon \R^{N} \to \bbC $ of rapid decay as $ x_1, \dots, x_N \to \infty $,
	the following asymptotic formula holds:
	\[
	\sum_{l \in \Z_{\ge 0}^N + \lambda} F(l) f(t(l+\alpha))
	\sim
	\varphi_{F, \alpha} \odot f (t)
	\quad \text{as } t \to +0.
	\]
\end{prop}

\begin{rem} \label{rem:phi_expression}
	In the situation of \cref{prop:asymp_lim}, when
	\[
	F(y) = C(y) P(y)
	\text{ with }
	C(y) = \prod_{1 \le i \le N} \delta(y_i \in a_i + k_i \Z_{\ge 0}), \,
	a_i, k_i \in \Z
	\text{ and }
	P(y) \in \bbC[y_1, \dots, y_N].
	\]
	we have
	\begin{align}
		\varphi_{F, u} (t_1, \dots, t_N)
		&=
		e^{t_1 u_1 + \cdots + t_N u_{N}}
		P \left( \frac{\partial}{\partial t_1}, \dots, \frac{\partial}{\partial t_N} \right)
		\frac{e^{a_1 t_1}}{1 - e^{k_1 t_1}} \cdots \frac{e^{a_N t_N}}{1 - e^{k_N t_N}} \\
		&\in \bbC ((t_1, \dots, t_N))[u_1, \dots, u_N].
	\end{align}
	Since any 
	\[
	F(y) \in 
	\delta(y \in \Z_{\ge 0}^N + \lambda) \cdot
	\bbC[ y_i, \delta(y_i \in a + k \Z) \mid 1 \le i \le N, a, k \in \Z ]
	\]
	can be written as a finite product of such $ C(y) P(y) $, we can regard $ \varphi_{F, u} (t_1, \dots, t_N) $ as an element of 
	$ \bbC ((t_1, \dots, t_N))[u_1, \dots, u_N] $.
\end{rem}

\begin{proof}[Proof of $ \cref{prop:asymp_lim} $]
	By the argument in \cref{rem:phi_expression}, we can assume 
	\[
	F(y) = C(y) P(y)
	\text{ with }
	C(y) = \prod_{1 \le i \le N} \delta(y_i \in a_i + k_i \Z_{\ge 0}), \,
	a_i, k_i \in \Z
	\text{ and }
	P(y) \in \bbC[y_1, \dots, y_N].
	\]
	Then, \cref{prop:Euler--Maclaurin_poly} implies the existence of 
	$ \varphi (t_1, \dots, t_N) \in \bbC ((t_1, \dots, t_N))[u_1, \dots, u_N] $
	which admits the asymptotic expansion
	\[
	\sum_{l \in \Z^N} F(l) f(t(l+\alpha))
	\sim
	\varphi \odot f (t)
	\quad \text{as } t \to +0
	\]
	for any $ C^\infty $ function $ f \colon \R^{N} \to \bbC $ of rapid decay as $ x_1, \dots, x_N \to \infty $.
	By applying the asymptotic expansion for $ f(x) \coloneqq e^{t_1 x_1 + \cdots + t_N x_N} $ with parameters $ t_1, \dots, t_N \in \R_{<0} $, we have
	\[
	\sum_{n \in \Z^N} F(n)
	e^{t t_1 (n_1 + u_1) + \cdots + t t_N (n_N + u_N)}
	\sim
	\varphi \odot f (t)
	\quad \text{as } t \to +0.
	\]
	This can be written as
	\[
	\varphi_{F, u} (t t_1, \dots, t t_N) 
	\sim
	\varphi (t t_1, \dots, t t_N) 
	\quad \text{as } t \to +0.
	\]
	Since both sides are meromorphic at $ t = 0 $, the above asymptotic expansion is, in fact, equality.
	Thus, we obtain the claim.
\end{proof}

%By this formula and \cref{rem:dominate} \cref{item:rem:dominate:1}, 
%To study asymptotic expansions of the GPPV invariants, it is efficient to give meromorphic functions which dominate asymptotic expansions of the GPPV invariants.
%This is the main idea of a proof in \cite{M_plumbed}.

% --------------------------------------------------------------------------

\subsection{An asymptotic formula related to meromorphic functions $ F_v(q) $} \label{subsec:asymp_F_v}

% --------------------------------------------------------------------------

We prepare a new asymptotic formula mentioned in \cref{item:proof_stepA}, which relates Laurent expansions of meromorphic functions and asymptotic expansions of infinite series that Cauchy principal values appear in its summands.
We use the notations which we prepared in \cref{subsec:coeff_of_F}.

\begin{prop} \label{prop:asymp_F_v}
	For a vector $ \nu \in \Q^V $, define a meromorphic function as
	\[
	F^{(\nu)} \left( (t_v)_{v \in V} \right)
	\coloneqq
	\prod_{v \in V} F_v \left( \bm{e} \left( \nu_v \right) e^{t_v} \right)
	\in \bbC(( t_v \mid v \in V )).
	\]
	Then, for any $ C^\infty $ function $ f \colon \R^{V} \to \bbC $ of rapid decay as $ \abs{x} \to \infty $,
	the following asymptotic formula holds:
	\[
	\sum_{l \in 2 \Z^V + \delta} 
	\bm{e}({}^t\!{\nu} l)
	F_l f(tl)
	\sim
	F^{(\nu)} \odot f (t)
	\quad  \text{ as } t \to +0.
	\]
\end{prop}

\begin{proof}
	For each $ v \in V $ and $ l_v \in \Z $, we can write
	\[
	F_{v, l_v}
	=
	\begin{dcases}
		-l_v \left( \delta (l_v = 1) + \delta (l_v = -1) \right)
		& \text{ if } \deg v = 1, \\
		\delta (l_v = 0)
		& \text{ if } \deg v = 2, \\
		\frac{1}{2} \frac{\delta (l_v \in \deg v + 2 \Z_{\ge -1})}{(\deg v - 3)!}
		\prod_{1 \le i \le \deg v - 3}
		\left( \frac{l_v - \deg v + 2}{2} + i \right)
		& \text{ if } \deg v \ge 3
	\end{dcases}
	\]
	by \cref{lem:coeff_of_F} \cref{item:lem:coeff_of_F:expression}.
	Thus, the map $ \Z_{\ge 0}^V + \delta \to \bbC; l \mapsto \bm{e}({}^t\!{\nu} l) F_l $ can be regarded as an element of
	\[
	\delta(y \in \Z_{\ge 0}^V + \delta) \cdot \bbC[ y_v, \delta(y_v \in a + k \Z) \mid v \in V, a, k \in \Z ].
	\]
	By \cref{lem:coeff_of_F} \cref{item:lem:coeff_of_F:expansion}, we have
	\[
	F^{(\nu)} \left( (t_v)_{v \in V} \right)
	=
	\prod_{v \in V} F_v \left( \bm{e} \left( \nu_v \right) e^{t_v} \right) 
	=
	2^{\abs{V_{\ge 3}}}
	\sum_{l = (l_v)_{v \in V} \in 2 \Z_{\ge -1}^{V} + \delta}
	\bm{e}({}^t\!{\nu} l)
	F_{l} e^{\sum_{v \in V} t_v l_v}.
	\]
	Thus, by \cref{prop:asymp_lim}, we have
	\begin{equation} \label{eq:asymp_F_v_cone}
		\sum_{l \in 2 \Z_{\ge -1}^V + \delta} 
		\bm{e}({}^t\!{\nu} l)
		F_l f(tl)
		\sim
		2^{-\abs{V_{\ge 3}}}
		F^{(\nu)} \odot f (t)
		\quad \text{as } t \to +0.
	\end{equation}
	
	For each $ v \in V $, let
	\[
	L_v \coloneqq
	\{ l_v \in \deg v + 2 \Z \mid F_{v, l_v} \neq 0 \}, \quad
	L_v^+ \coloneqq
	L_v \cap (\deg v + 2 \Z_{\ge -1}).
	\]
	By \cref{lem:coeff_of_F} \cref{item:lem:coeff_of_F:expansion}, we have
	\begin{align}
		L_v &=
		\begin{dcases}
			\{ \pm 1 \}
			& \text{ if } \deg v = 1, \\
			\{ 0 \}
			& \text{ if } \deg v = 2, \\
			(\deg v + 2 \Z_{\ge -1}) \sqcup (-\deg v - 2 \Z_{\ge -1})
			& \text{ if } \deg v \ge 3,
		\end{dcases}
		\\
		L_v^+ &=
		\begin{dcases}
			\{ \pm 1 \}
			& \text{ if } \deg v = 1, \\
			\{ 0 \}
			& \text{ if } \deg v = 2, \\
			\deg v + 2 \Z_{\ge -1}
			& \text{ if } \deg v \ge 3,
		\end{dcases}
	\end{align}
	Thus,
	\[
	\begin{array}{ccc}
		\displaystyle \{ \pm 1 \}^V \times \prod_{v \in V} L_v^+ & \longrightarrow & \displaystyle \prod_{v \in V} L_v \\
		(\varepsilon, l) & \longmapsto & \varepsilon l \coloneqq (\varepsilon_v l_v)_{v \in V}
	\end{array}
	\]
	are surjective and its inverse image of $ \varepsilon l $ is
	$ \{ (\varepsilon \varepsilon', \varepsilon' l) \mid \varepsilon' \in \{ \pm 1  \}^{V_{\le 2}} \} $,
	where $ V_{\le 2} \coloneqq \{ v \in V \mid \deg v \le 2 \} $.
	Let $ f_\varepsilon (x) \coloneqq f(\varepsilon x) $.
	By the above $ 2^{\abs{V_{\le 2}}} $-to-$ 1 $ map, we have
	\begin{align}
		\sum_{l \in 2 \Z^V + \delta} 
		\bm{e}({}^t\!{\nu} l)
		F_l f(tl)
		&=
		2^{-\abs{V_{\le 2}}}
		\sum_{\varepsilon \in \{ \pm 1 \}^V}
		\sum_{l \in 2 \Z_{\ge -1}^V + \delta} 
		\bm{e}({}^t\!{\nu} \varepsilon l)
		F_{\varepsilon l} f(t \varepsilon l) \\
		&=
		2^{-\abs{V_{\le 2}}}
		\sum_{\varepsilon \in \{ \pm 1 \}^V}
		\left( \prod_{v \in V} \varepsilon_v^{\deg v} \right)
		\sum_{l \in 2 \Z_{\ge -1}^V + \delta} 
		\bm{e}({}^t\!(\varepsilon \nu) l)
		F_{l} f_\varepsilon(tl)
		\quad
		\text{ by \cref{lem:coeff_of_F} \cref{item:lem:coeff_of_F:symmetry}} \\
		&\sim
		2^{-\abs{V}}
		\sum_{\varepsilon \in \{ \pm 1 \}^V}
		\left( \prod_{v \in V} \varepsilon_v^{\deg v} \right)
		F^{(\varepsilon \nu)} \odot f_\varepsilon (t)
		\text{ as } t \to +0
		\quad
		\text{ by \cref{eq:asymp_F_v_cone}} \\
		&=
		2^{-\abs{V}}
		\sum_{\varepsilon \in \{ \pm 1 \}^V}
		\sum_{m \in \Z^{V}}
		\left( \prod_{v \in V} \varepsilon_v^{\deg v} \right)
		\left( F^{(\varepsilon \nu)}( (\varepsilon_v t_v)_{v \in V} ) \right) \odot f(t)
		\text{ as } t \to +0\\
		&=
		F^{(\nu)} \odot f (t)
		\text{ as } t \to +0.
	\end{align}
	Here, the last equality follows from the symmetry $ F_v (z_v^{-1}) = (-1)^{\deg v} F_v (z_v) $.
\end{proof}

% --------------------------------------------------------------------------

\subsection{The meromorphic functions which relate to asymptotic expansions of the GPPV invariants} \label{subsec:asymp_HB}

% --------------------------------------------------------------------------

Finally, we give the meromorphic functions whose Laurent expansion relates to asymptotic expansions of the GPPV invariants.
Define a meromorphic function as
\begin{align}
	\varphi_{\Gamma, k} \left( (t_v)_{v \in V} \right)
	&\coloneqq
	\sum_{\mu \in \Z^V/2k\Z^V}
	\bm{e} \left( \frac{1}{4k} {}^t\!\mu W \mu \right)
	\prod_{v \in V} F_v \left( \zeta_{2k}^{\mu_v} e^{t_v} \right)
	\\
	&=
	\sum_{\mu \in \Z^V/2k\Z^V}
	\bm{e} \left( \frac{1}{4k} {}^t\!\mu W \mu \right)
	F^{(\mu/2k)} \left( (t_v)_{v \in V} \right)
	\in \bbC(t_v \mid v \in V).
\end{align}
%We remark that
%\[
%B_m(\Gamma, k)
%=
%\sum_{\mu \in \Z^V/2k\Z^V}
%\bm{e} \left( \frac{1}{4k} {}^t\!\mu W \mu \right)
%B_m^{(\mu/4k)}.
%\]

\cref{prop:asymp_F_v} implies the following proposition. 
By this proposition, the proof of our main theorem (\cref{thm:main}) is reduced to examining the properties of the meromorphic function $ \varphi_{\Gamma, k} \left( (t_v)_{v \in V} \right) $.

\begin{prop} \label{prop:F_dominates_HB}
	\begin{align}
		&\sum_{\substack{
				a \in \Z^V/W(\Z^V), \\
				b \in ( 2\Z^V + \delta ) /2W(\Z^V)
		}}
		\bm{e} \left( -k {}^t\!a W^{-1} a - {}^t\!a W^{-1} b \right)
		\restrict{\widehat{Z}_{b} (q; M(\Gamma)) q^{\sum_{v \in V} (w_{v} + 3)/4}}{q = \zeta_k e^{-t}}
		\\
		\sim \,
		&\frac{ \zeta_8^{\abs{V}} \sqrt{\abs{\det W}} }{\sqrt{2k}^{\abs{V}}}
		\varphi_{\Gamma, k} \odot \left( e^{-Q(x)/4} \right) (\sqrt{t})
		\quad \text{as } t \to +0.
	\end{align}
\end{prop}

\begin{proof}
	By \cref{eq:GPPV:expression}, we can write the left hand side of the claim as
	\[
	\sum_{l \in 2\Z^V + \delta}
	\left(
	\sum_{a \in \Z^V / W(\Z^V)}
	\bm{e} \left( k Q \left( a + \frac{1}{2k} l \right) \right)
	\right)
	F_l e^{-Q(\sqrt{t} l)/4}.
	\]
	By reciprocity of Gauss sums (\cref{prop:reciprocity}), we have
	\[
	\sum_{a \in \Z^V / W(\Z^V)}
	\bm{e} \left( k Q \left( a + \frac{1}{2k} l \right) \right)
	=
	\frac{ \zeta_8^{\abs{V}} \sqrt{\abs{\det W}} }{\sqrt{2k}^{\abs{V}}}
	\sum_{\mu \in \Z^V/2k\Z^V}
	\bm{e} \left( \frac{1}{4k} {}^t\!\mu W \mu + \frac{1}{2k} {}^t\!\mu l \right).
	\]
	Thus, we have
	\begin{align}
		&\phant
		\sum_{\substack{
				a \in \Z^V/W(\Z^V), \\
				b \in ( 2\Z^V + \delta ) /2W(\Z^V)
		}}
		\bm{e} \left( -k {}^t\!a W^{-1} a - {}^t\!a W^{-1} b \right)
		\restrict{\widehat{Z}_{b} (q; M(\Gamma)) q^{\sum_{v \in V} (w_{v} + 3)/4}}{q = \zeta_k e^{-t}}
		\\
		&=
		\frac{ \zeta_8^{\abs{V}} \sqrt{\abs{\det W}} }{\sqrt{2k}^{\abs{V}}}
		\sum_{\mu \in \Z^V/2k\Z^V}
		\bm{e} \left( \frac{1}{4k} {}^t\!\mu W \mu \right)
		\sum_{l \in 2\Z^V + \delta}
		\bm{e} \left( \frac{1}{2k} {}^t\!\mu l \right)	
		F_l e^{-Q(\sqrt{t} l)/4}
		\\
		&\sim
		\frac{ \zeta_8^{\abs{V}} \sqrt{\abs{\det W}} }{\sqrt{2k}^{\abs{V}}}
		\sum_{\mu \in \Z^V/2k\Z^V}
		\bm{e} \left( \frac{1}{4k} {}^t\!\mu W \mu \right)
		\left(
		F^{(\mu/2k)} \odot \left( e^{-Q(x)/4} \right) (\sqrt{t})
		\right)
		\quad \text{as } t \to +0
		\quad \text{ by \cref{prop:asymp_F_v}}
		\\
		&=
		\frac{ \zeta_8^{\abs{V}} \sqrt{\abs{\det W}} }{\sqrt{2k}^{\abs{V}}}
		\varphi_{\Gamma, k} \odot \left( e^{-Q(x)/4} \right) (\sqrt{t})
		\quad \text{as } t \to +0.
	\end{align}
\end{proof}

% --------------------------------------------------------------------------

\section{Meromorphic functions obtained by pruning plumbed graphs} \label{sec:pruned}

% --------------------------------------------------------------------------

In this section, we complete the proof of the main theorem.

% --------------------------------------------------------------------------

\subsection{Main result} \label{subsec:main_result}

% --------------------------------------------------------------------------

Define Laurent coefficients of the meromorphic function $ \varphi_{\Gamma, k} \left( (t_v)_{v \in V} \right) $ as
\[
\sum_{m \in \Z^V} B_m(\Gamma, k) \prod_{v \in V} t_v^{m_v}
\coloneqq
\varphi_{\Gamma, k} \left( (t_v)_{v \in V} \right).
\]
Our aim in this section is to prove the following theorem.

\begin{thm} \label{thm:phi_properties}
	\begin{enumerate}
		\item \label{item:thm:phi_properties:non-negative}
		For $ m \in \Z^V $, if $ B_m (\Gamma, k) \neq 0 $ then $ \sum_{v \in V} m_v \ge 0 $.
		\item \label{item:thm:phi_properties:constant_only}
		For $ m \in \Z^V $, if $ B_m (\Gamma, k) \neq 0 $ and $ \sum_{v \in V} m_v = 0 $ then $ m=0 $.
		\item \label{item:thm:phi_properties:constant_expression}
		\begin{equation} \label{eq:F(0)}
			B_0 (\Gamma, k)
			=
			\sum_{\mu \in (\Z \smallsetminus k\Z)^V/2k\Z^V}
			\bm{e} \left( \frac{1}{4k} {}^t\!\mu W \mu \right)
			\prod_{v \in V} F_v \left( \zeta_{2k}^{\mu_v} \right).
		\end{equation}
	\end{enumerate}
\end{thm}

As will be shown below, this theorem implies our main theorem (\cref{thm:main}).

\begin{proof}[Proof of $ \cref{thm:main} $]
	We have
	\[
	\varphi_{\Gamma, k} \odot \left( e^{-Q(x)/4} \right) (t)
	\in \bbC[[t]]
	\]
	by \cref{thm:phi_properties} \cref{item:thm:phi_properties:non-negative} and	
	\[
	\varphi_{\Gamma, k} \odot \left( e^{-Q(x)/4} \right) (0)
	=
	B_0 (\Gamma, k)
	=
	\sum_{\mu \in (\Z \smallsetminus k\Z)^V/2k\Z^V}
	\bm{e} \left( \frac{1}{4k} {}^t\!\mu W \mu \right)
	\prod_{v \in V} F_v \left( \zeta_{2k}^{\mu_v} \right)
	\]
	by \cref{thm:phi_properties} \cref{item:thm:phi_properties:constant_only,item:thm:phi_properties:constant_expression}.
	Thus, by \cref{prop:F_dominates_HB}, the limit
	\[
	\lim_{t \to +0}
	\sum_{\substack{
			a \in \Z^V/W(\Z^V), \\
			b \in ( 2\Z^V + \delta ) /2W(\Z^V)
	}}
	\bm{e} \left( -k {}^t\!a W^{-1} a - {}^t\!a W^{-1} b \right)
	\restrict{\widehat{Z}_{b} (q; M(\Gamma))}{q = \zeta_k e^{-t}}
	\]
	converges and equals to
	\[
	\frac{ \zeta_8^{\abs{V}} \zeta_{4k}^{-\sum_{v \in V} (w_{v} + 3)} \sqrt{\abs{\det W}} }{\sqrt{2k}^{\abs{V}}}
	\sum_{\mu \in (\Z \smallsetminus k\Z)^V/2k\Z^V}
	\bm{e} \left( \frac{1}{4k} {}^t\!\mu W \mu \right)
	\prod_{v \in V} F_v \left( \zeta_{2k}^{\mu_v} \right),
	\]
	which equals to $ \WRT_k (M(\Gamma))\cdot 2 (\zeta_{2k} - \zeta_{2k}^{-1}) \sqrt{\abs{\det W}} $ by \cref{prop:WRT_rep}.
	Thus, we obtain the claim.
\end{proof}

We will prove \cref{thm:phi_properties} by induction on a sequence of trees obtained by repeating ``pruning trees.''
For the reader's convenience, before giving the complex proof of the general case, we will first provide a proof for a simple case in the next subsection.

% --------------------------------------------------------------------------

\subsection{A simple example} \label{subsec:example}

% --------------------------------------------------------------------------

In this subsection, we prove \cref{thm:phi_properties} when the plumbed graph $ \Gamma $ is the Y-graph shown in \cref{fig:Y-graph}.

\begin{figure}[htb]
	\centering
	\begin{tikzpicture}
		\node[shape=circle,fill=black, scale = 0.4] (1) at (0,0) { };
		\node[shape=circle,fill=black, scale = 0.4] (2) at (0.866,0.5) { }; % (√3/2, 1/2)
		\node[shape=circle,fill=black, scale = 0.4] (3) at (-0.866,0.5) { };
		\node[shape=circle,fill=black, scale = 0.4] (4) at (0,-1) { };
		
		\node[draw=none] (B1) at (0.4,-0.1) {$ w_1 $};
		\node[draw=none] (B2) at (0.866+0.4, 0.5) {$ w_2 $};
		\node[draw=none] (B3) at (-0.866-0.4,0.5) {$ w_3 $};
		\node[draw=none] (B4) at (0.4,-1) {$ w_4 $};
		
		\path [-](1) edge node[left] {} (2);
		\path [-](1) edge node[left] {} (3);
		\path [-](1) edge node[left] {} (4);
	\end{tikzpicture}
	\caption{A Y-graph} \label{fig:Y-graph}
\end{figure}
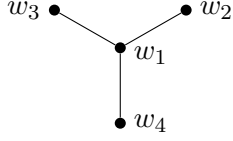

In this case, we can write
\[
\varphi_{\Gamma, k} \left( t_1, t_2, t_3, t_4 \right)
=
\sum_{\mu \in \Z^4/2k\Z^4}
\bm{e} \left( \frac{1}{4k} {}^t\!\mu W \mu \right)
\frac{F_0 \left( \zeta_{2k}^{\mu_2} e^{t_2} \right) F_0 \left( \zeta_{2k}^{\mu_3} e^{t_3} \right) F_0 \left( \zeta_{2k}^{\mu_4} e^{t_4} \right)}
{F_0 \left( \zeta_{2k}^{\mu_1} e^{t_1} \right)}
\in \bbC[[t_2, t_3, t_4]]((t_1)),
\]
where $ F_0(z) \coloneqq z - z^{-1} $.
Since $ F_0 \left( \zeta_{2k}^{\mu_1} \right) \neq 0 $ for any $ \mu_1 \in \Z \smallsetminus k\Z $, we have
\[
\varphi_{\Gamma, k} \left( t_1, t_2, t_3, t_4 \right)
\equiv
\sum_{\mu \in \Z^4/2k\Z^4, \, \mu_1 \in k\Z}
\bm{e} \left( \frac{1}{4k} {}^t\!\mu W \mu \right)
\frac{F_0 \left( \zeta_{2k}^{\mu_2} e^{t_2} \right) F_0 \left( \zeta_{2k}^{\mu_3} e^{t_3} \right) F_0 \left( \zeta_{2k}^{\mu_4} e^{t_4} \right)}
{F_0 \left( \zeta_{2k}^{\mu_1} e^{t_1} \right)}
\bmod \bbC[[t_1, t_2, t_3, t_4]].
\]
Since $ F_0 \left( \zeta_{2k}^{\mu_1} e^{t_1} \right) = (-1)^{\mu_1/k}(e^{t_1} - e^{-t_1}) $ for any $ \mu_1 \in k\Z $, we obtain
\begin{align}
	&\phant
	\varphi_{\Gamma, k} \left( t_1, t_2, t_3, t_4 \right)
	\\
	&\equiv
	\frac{1}{e^{t_1} - e^{-t_1}}
	\sum_{\mu \in \Z^4/2k\Z^4, \, \mu_1 \in k\Z}
	\bm{e} \left( \frac{1}{4k} {}^t\!\mu W \mu + \frac{\mu_1}{2k} \right)
	F_0 \left( \zeta_{2k}^{\mu_2} e^{t_2} \right) F_0 \left( \zeta_{2k}^{\mu_3} e^{t_3} \right) F_0 \left( \zeta_{2k}^{\mu_4} e^{t_4} \right)
	\bmod \bbC[[t_1, t_2, t_3, t_4]].
\end{align}
The right hand side is equal to
\[
\frac{1}{e^{t_1} - e^{-t_1}}
\sum_{\mu_1 \in \Z/2\Z}
\bm{e} \left( \frac{k}{4} w_1 \mu_1^2 + \frac{1}{2} \mu_1 \right)
\prod_{2 \le i \le 4}
G_i (\mu_i, t_i),
\]
where
\[
G_i (\mu_1, t_i)
\coloneqq
\sum_{\mu_i \in \Z/2k\Z}
\bm{e} \left( \frac{1}{4k} w_i \mu_i^2 + \frac{1}{2} \mu_1 \mu_i \right)
F_0 \left( \zeta_{2k}^{\mu_i} e^{t_i} \right)
\in \bbC[[t_i]].
\]
By taking $ \mu_i \mapsto -\mu_i $, we have $ G_i (\mu_1, t_i) = G_i (\mu_1, -t_i) $.
Thus, we obtain $ G_i (\mu_1, t_i) \in t_i \bbC[[t_i]] $.
Therefore, it holds that
\begin{equation} \label{eq:Y-graph_phi}
	\varphi_{\Gamma, k} \left( t_1, t_2, t_3, t_4 \right)
	\in 
	\bbC[[t_1, t_2, t_3, t_4]] + \frac{t_2 t_3 t_4}{t_1} \bbC[[t_1, t_2, t_3, t_4]].
\end{equation}
Here we remark that $ \varphi_{\Gamma, k} \left( t_1, t_2, t_3, t_4 \right) \notin \bbC[[t_1, t_2, t_3, t_4]] $.
Nevertheless, for $ m \in \Z^V $ with $ B_m \neq 0 $, we have $ \sum_{v \in V} m_v \ge 0 $ and $ m=0 $ if $ \sum_{v \in V} m_v = 0 $ by \cref{eq:Y-graph_phi}.
Thus, we have $ B_0 = \restrict{\varphi_{\Gamma, k} \left( t, t, t, t \right)}{t = 0} $.
Since
\[
\frac{1}{e^{t} - e^{-t}}
\sum_{\mu_1 \in \Z/2\Z}
\bm{e} \left( \frac{k}{4} w_1 \mu_1^2 + \frac{1}{2} \mu_1 \right)
\prod_{2 \le i \le 4}
G_i (\mu_i, t)
\in t \bbC[[t]]
\]
and $ F_0 \left( \zeta_{2k}^{\mu} \right) = 0 $ for any $ \mu \in k\Z $, we have
\begin{align}
	B_0
	&=
	\restrict{\varphi_{\Gamma, k} \left( t, t, t, t \right)}{t = 0} \\
	&=
	\sum_{\mu \in \Z^4/2k\Z^4, \, \mu_1 \notin k\Z}
	\bm{e} \left( \frac{1}{4k} {}^t\!\mu W \mu \right)
	\frac{F_0 \left( \zeta_{2k}^{\mu_2} \right) F_0 \left( \zeta_{2k}^{\mu_3} \right) F_0 \left( \zeta_{2k}^{\mu_4} \right)}
	{F_0 \left( \zeta_{2k}^{\mu_1} \right)} \\
	&=
	\sum_{\mu \in (\Z \smallsetminus k\Z)^4/2k\Z^4}
	\bm{e} \left( \frac{1}{4k} {}^t\!\mu W \mu \right)
	\frac{F_0 \left( \zeta_{2k}^{\mu_2} \right) F_0 \left( \zeta_{2k}^{\mu_3} \right) F_0 \left( \zeta_{2k}^{\mu_4} \right)}
	{F_0 \left( \zeta_{2k}^{\mu_1} \right)}.
\end{align}
Thus, we obtain a proof of \cref{thm:phi_properties} for the case of Y-graphs.

% --------------------------------------------------------------------------

\subsection{Pruned graphs} \label{subsec:pruned_graph}

% --------------------------------------------------------------------------

We define ``pruning graphs'' as follows.

\begin{dfn}
	Let $ \Gamma $ be a weighted graph, $ V $ be the set of its vertices, $ w_v $ be the weight of each vertex $ v \in V $.
	Then, we define the \textbf{degree} of each vertex $ v \in V $ as
	\begin{align}
		\deg(v) = \deg_\Gamma (v)
		\coloneqq
		\# \left\{ v' \in V \relmiddle{|} v' \text{ is adjacent to } v \right\}
	\end{align}
	and let
	\begin{align}
		V_1 &\coloneqq \left\{ v \in V \relmiddle{|} \deg(v) = 1 \right\}, 
		\\
		V^{\sprod{1}} &\coloneqq \left\{ v \in V \relmiddle{|} \deg(v) \ge 2 \right\}.
	\end{align}
	Moreover, for each vertex $ v \in V^{\sprod{1}} $, let
	\begin{align}
		\overline{v} 
		&\coloneqq
		\left\{ i \in V_1 \relmiddle{|} i \text{ is adjacent to } v \right\},
		\\		
		w_v^{\sprod{1}}
		&\coloneqq
		w_v - \sum_{i \in \overline{v}} \frac{1}{w_i}.
	\end{align}
	Then, we define \textbf{the pruned graph} $ \Gamma^{\sprod{1}} $ of $ \Gamma $ as the weighted graph obtained by restricting the vertex set of $ \Gamma $ to $ V^{\sprod{1}} $ and assigning weight $ w_v^{\sprod{1}} $ to each vertex $ v \in V^{\sprod{1}} $.
\end{dfn}

Repeating pruning can reduce any plumbed graphs to trees with either one or two vertices.
That is our main idea to prove our main theorem.

The following lemma is useful to study the properties of pruned graph.

\begin{lem}[{\cite[Lemma 2.3]{M_plumbed}}] \label{lem:block_matrix}
	For a symmetric block matrix
	\[
	X = \pmat{A & B \\ {}^t\!B & C} \in \GL_{m+n} (\bbC)
	\]
	such that $ A \in \GL_m(\bbC) $ and $ C \in \GL_n(\bbC) $ be symmetric matrices and $ B \in \Mat_{m, n}(\bbC) $, let
	\[
	S \coloneqq (C - {}^t\!B A^{-1} B)^{-1} \in \GL_n(\bbC), \quad
	T \coloneqq -A^{-1} B \in \Mat_{m, n}(\bbC).
	\]
	Then, it holds
	\[
	X^{-1} = \pmat{T \\ I} S \pmat{{}^t\!T & I} + \pmat{A^{-1} & \\ & O}, \quad
	\det S = \frac{\det A}{\det X}.
	\]	
	In particular, $ S $ is the $ n \times n $ bottom right submatrix of $ X^{-1} $.
\end{lem}

Pruned graphs have the following properties.

\begin{rem} \label{rem:pruned}
	\begin{enumerate}
		\item \label{item:rem:pruned:deg}
		For each vertex $ v \in V^{\sprod{1}} $, it holds $ \deg_{\Gamma^{\sprod{1}}} (v) = \deg (v) - \abs{\overline{v}} $.
		\item \label{item:rem:pruned:block}
		Let $ W $ and $ W^{\sprod{1}} $ be the adjacency matrices of $ \Gamma $ and $ \Gamma^{\sprod{1}} $ respectively.
		Then, $ \left( W^{\sprod{1}} \right)^{-1} $ is $ V^{\sprod{1}} \times V^{\sprod{1}} $-submatrix of $ W^{-1} $ and it holds
		\[
		\det W^{\sprod{1}} = \det W \prod_{i \in V_1} \frac{1}{w_i}
		\]	
		by the following \cref{lem:block_matrix}. 	
		In particular, if $ W $ is negative definite, then so is $ W^{\sprod{1}} $. 
		\item \label{item:rem:pruned:quad_form}
		For $ x = (x_v)_{v \in V} \in \R^V $, it holds
		\[
		\transpose{x} W x = \transpose{x^{\sprod{1}}} W^{\sprod{1}} x^{\sprod{1}} + \sum_{i \in V_1} \frac{x_i^2}{w_i}, \quad
		\text{ where }
		x^{\sprod{1}} \coloneqq \left( x_v - \sum_{i \in \overline{v}} \frac{x_i^2}{w_i} \right)_{v \in V^{\sprod{1}}} \in \R^{V^{\sprod{1}}}.
		\]
	\end{enumerate}
\end{rem}

The principle in \cref{rem:pruned} \cref{item:rem:pruned:block} is employed in the technique for calculating the determinants of the linking matrices of plumbed graphs, as described in \cite[Section 21]{Eisenbud-Neumann}.

% --------------------------------------------------------------------------

\subsection{Meromorphic functions obtained by pruning plumbed graphs} \label{subsec:pruned_rat_func}

% --------------------------------------------------------------------------

Let us introduce notation for a sequence of plumbed graphs determined by pruning the plumbed graph $ \Gamma $.

\begin{nota*}
	For a non-negative integer $ n $, define the following notations. 
	\begin{itemize}
		\item For $ n=0 $, let $ \Gamma^{\sprod{0}} \coloneqq \Gamma $ and for $ n \ge 1 $, define $ \Gamma^{\sprod{n}} $ as
		$ \Gamma^{\sprod{n}} \coloneqq (\Gamma^{\sprod{n-1}})^{\sprod{1}} $ inductively. 
		\item Let $ V^{\sprod{n}} $ be the set of vertices of $ \Gamma^{\sprod{n}} $. 
		\item Let $ W^{\sprod{n}} $ be the adjacency matrices of $ \Gamma^{\sprod{n}} $. 
		\item For each vertex $ v \in V^{\sprod{n}} $, define the following notations. 
		\begin{itemize}
			\item Let $ w_v^{\sprod{n}} $ be the weight of $ v $. 
			\item Let $ \overline{v}^{\sprod{n}}
			\coloneqq
			\left\{ i \in \Gamma^{\sprod{n}} \relmiddle{|} i \text{ is adjacent to } v \text{ and } \deg_{\Gamma^{\sprod{n}}} (i) = 1 \right\} $.
			\item Let $ V_v^{\sprod{n}} $ be the set of vertices of the connected component of
			$ \Gamma \smallsetminus (\Gamma^{\sprod{n}} \smallsetminus \{ v \}) $
			containing $ v $.
			We can express $ V_v^{\sprod{n}} $ inductively as
			\[
			V_v^{\sprod{0}} = \{ v \}, \quad
			V_v^{\sprod{n}} = \{ v \} \sqcup \coprod_{i \in \overline{v}^{\sprod{n-1}}} V_i^{\sprod{n-1}}.
			\]
			\item For $ n=0 $, let $ M_v^{\sprod{0}} \coloneqq 1 $ and $ \widetilde{w}_v^{\sprod{0}} \coloneqq w_v \in \Z $ and 
			for $ n \ge 1 $, define $ M_v^{\sprod{n}} \in \Z $ and $ \widetilde{w}_v^{\sprod{n}} \in \Z $ inductively as
			\begin{align}
				M_v^{\sprod{n}}
				&\coloneqq
				M_v^{\sprod{n-1}}
				\prod_{i \in \overline{v}^{\sprod{n-1}}} \widetilde{w}_i^{\sprod{n-1}},
				\\
				\widetilde{w}_v^{\sprod{n}} 
				&\coloneqq
				M_v^{\sprod{n}} w_v^{\sprod{n}}.
			\end{align}
		\end{itemize}
		\item Let $ M^{\sprod{n}} \coloneqq \mathrm{diag} \left( M_v^{\sprod{n}} \right)_{v \in V^{\sprod{n}}} \in \End ( \Z^{V^{\sprod{n}}} ) $. 
		\item Let $ t_V \coloneqq \left( t_v \right)_{v \in V} $ be complex variables. 
		\item For each vertex $ v \in V^{\sprod{n}} $, define the following notations. 
		\begin{itemize}
			\item Let $ t_v^{\sprod{n}} \coloneqq \left( t_i \right)_{i \in V_v^{\sprod{n}}} $ be complex variables. 
			\item For $ \mu \in \Z $, define $ F_v^{\sprod{n}} (\mu, t_v^{\sprod{n}}) $ and $ G_v^{\sprod{n}} (\mu, t_v^{\sprod{n}}) $ inductively as
			\begin{align}
				F_v^{\sprod{0}} (\mu, t_v^{\sprod{0}}) 
				&\coloneqq 
				F_v \left( \zeta_{2k}^{\mu} e^{t_v} \right),
				\\
				F_v^{\sprod{n}} (\mu, t_v^{\sprod{n}})
				&\coloneqq
				F_v^{\sprod{n-1}} (\mu, t_v^{\sprod{n-1}})
				\prod_{i \in \overline{v}^{\sprod{n-1}}} G_i^{\sprod{n-1}} (\mu, t_i^{\sprod{n-1}}),
				\\
				G_v^{\sprod{n}} (\mu, t_v^{\sprod{n}})
				&\coloneqq
				\bm{e} \left( \frac{\mu^2}{4k w_v^{\sprod{n}}} \right)
				\sum_{\mu_v \in \Z / 2k M_v^{\sprod{n}} \Z}
				\bm{e} \left( \frac{1}{4k} \left( w_v^{\sprod{n}} \mu_v^2 + 2 \mu \mu_v \right) \right)
				F_v^{\sprod{n}} (\mu_v, t_v^{\sprod{n}}).
			\end{align}
		\end{itemize}
	\end{itemize}
\end{nota*}

\begin{rem} \label{rem:pruned_w}
	By induction and \cref{rem:pruned} \cref{item:rem:pruned:block}, it holds $ w_i^{\sprod{n}} < 0 $ for any $ n \in \Z_{\ge 0} $ and $ v \in V^{\sprod{n}} $.
\end{rem}

Well-definedness of $ F_v^{\sprod{n}} (\mu, t_v^{\sprod{n}}) $ and $ G_v^{\sprod{n}} (\mu, t_v^{\sprod{n}}) $ follows from the following lemma.

\begin{lem} \label{lem:F_invariance}
	For $ v \in V^{\sprod{n}} $ and $ \mu \in \Z $, it holds that
	\begin{align}
		F_v^{\sprod{n}} (\mu + 2k M_v^{\sprod{n}}, t_v^{\sprod{n}})
		&=
		F_v^{\sprod{n}} (\mu, t_v^{\sprod{n}}),
		\\ 
		G_v^{\sprod{n}} (\mu + 2k \widetilde{w}_v^{\sprod{n}}, t_v^{\sprod{n}})
		&=
		G_v^{\sprod{n}} (\mu, t_v^{\sprod{n}}).
	\end{align}
\end{lem}

\begin{proof}
	For $ F_v^{\sprod{0}} $, we have the claim because 
	$ F_v^{\sprod{0}} (\mu, t_v^{\sprod{0}}) = ( \zeta_{2k}^{\mu} e^{t_v} - \zeta_{2k}^{-\mu} e^{-t_v} )^{2 - \deg v} $ 
	is invariant under $ \mu \mapsto \mu + 2k $.
	Suppose that the claim holds for $ F_v^{\sprod{n}} $.
	Then, we obtain the claim for $ G_v^{\sprod{n}} $ since
	$ \widetilde{w}_v^{\sprod{n}} = M_v^{\sprod{n}} w_v^{\sprod{n}} $.
	Moreover, we have the claim for $ F_v^{\sprod{n+1}} $ since
	$ M_v^{\sprod{n+1}} = M_v^{\sprod{n}} \prod_{i \in \overline{v}^{\sprod{n+1}}} \widetilde{w}_i^{\sprod{n}} $.	
\end{proof}

By induction and \cref{rem:pruned} \cref{item:rem:pruned:deg}, we obtain the following lemma.

\begin{lem} \label{lem:F_symmetry}
	For $ v \in V^{\sprod{n}} $ and $ \mu \in \Z $, it holds that
	\[
	F_v^{\sprod{n}} (-\mu, -t_v^{\sprod{n}})
	=
	(-1)^{\deg_{\Gamma^{\sprod{n}}} v} F_v^{\sprod{n}} (-\mu, -t_v^{\sprod{n}}), \quad
	G_v^{\sprod{n}} (-\mu, -t_v^{\sprod{n}})
	=
	(-1)^{\deg_{\Gamma^{\sprod{n}}} v} G_v^{\sprod{n}} (-\mu, -t_v^{\sprod{n}}).
	\]
\end{lem}

Below, we present three lemmas that are key to the proof \cref{thm:phi_properties}.

\begin{lem} \label{lem:pruned_F}
	For any $ n \in \Z_{\ge 0} $ with $ V^{\sprod{n}} \neq \emptyset $, it holds that
	\[
	\varphi_{\Gamma, k} ( t_V )
	=
	\frac{1}{\abs{\det M^{\sprod{n}}}}
	\sum_{\left. 
		\mu \in \Z^{V^{\sprod{n}}} 
		\middle/
		2k M^{\sprod{n}} ( \Z^{V^{\sprod{n}}} ) 
		\right. }
	\bm{e} \left( \frac{1}{4k} {}^t\!\mu W^{\sprod{n}} \mu \right)
	\prod_{v \in V^{\sprod{n}}} F_v^{\sprod{n}} ( \mu_v, t_v^{\sprod{n}} ).
	\]
\end{lem}

\begin{proof}
	By induction on $ n $.
	When $ n=0 $, the claim follows from the definition.
	Suppose that the claim holds for $ n $.
	Let $ W_{\ge 2}^{\sprod{n}} $,$ M_{\ge 2}^{\sprod{n}} $ be the $ V^{\sprod{n+1}} \times V^{\sprod{n+1}} $-components of matrices $ W^{\sprod{n}} $ and $ M^{\sprod{n}} $.
	Then, we have
	\begin{align}
		\varphi_{\Gamma, k} ( t_V )
		= \,
		&\frac{1}{\abs{\det M^{\sprod{n}}}} 
		\sum_{\left. 
			\mu \in \Z^{V^{\sprod{n+1}}} 
			\middle/
			2k M_{\ge 2}^{\sprod{n}} ( \Z^{V^{\sprod{n+1}}} ) \right. }
		\bm{e} \left( \frac{1}{4k} {}^t\!\mu W_{\ge 2}^{\sprod{n}} \mu \right)
		\prod_{v \in V^{\sprod{n+1}}} F_v^{\sprod{n}} ( \mu_v, t_v^{\sprod{n}} )
		\\
		&\prod_{i \in \overline{v}^{\sprod{n}}} \sum_{\mu_i \in \Z/2k M_i^{\sprod{n}} \Z}
		\bm{e} \left( \frac{1}{4k} \left( w_i^{\sprod{n}} \mu_i^2 + 2 \mu_v \mu_i \right) \right)
		F_i^{\sprod{n}} ( \mu_i, t_i^{\sprod{n}} )	
		\\
		= \,
		&\frac{1}{\abs{\det M^{\sprod{n}}}} 
		\sum_{\left. 
			\mu \in \Z^{V^{\sprod{n+1}}} 
			\middle/
			2k M_{\ge 2}^{\sprod{n}} ( \Z^{V^{\sprod{n+1}}} ) \right. }
		\bm{e} \left( \frac{1}{4k} {}^t\!\mu W_{\ge 2}^{\sprod{n}} \mu \right)
		\prod_{v \in V^{\sprod{n+1}}} F_v^{\sprod{n}} ( \mu_v, t_v^{\sprod{n}} )
		\\
		&\prod_{i \in \overline{v}^{\sprod{n}}}
		\bm{e} \left( -\frac{\mu_v^2}{4k w_i^{\sprod{n}}} \right)
		G_i^{\sprod{n}} ( \mu_v, t_i^{\sprod{n}} ).
	\end{align}
	Since $ F_i^{\sprod{n}} \left( \mu_v, t_i \right) $ is determined for $ \mu_v \bmod 2k M_i^{\sprod{n}} $ by \cref{lem:F_invariance}, it is determined also for $ \mu_v \bmod 2k M_v^{\sprod{n+1}} $.
	Since also $ \bm{e} \left( -\mu_v^2 / 4k w_i^{\sprod{n}} \right) $  is determined for $ \mu_v \bmod 2k M_v^{\sprod{n+1}} $, we obtain
	\begin{align}
		&\varphi_{\Gamma, k} ( t_V )
		\\
		= \,
		&\frac{1}{\abs{\det M^{\sprod{n}}}} \left( \prod_{i \in V_1^{\sprod{n}}} \abs{\widetilde{w}_i^{\sprod{n}}}^{-1} \right)
		\sum_{\left. 
			\mu \in \Z^{V^{\sprod{n+1}}} 
			\middle/
			2k M^{\sprod{n+1}} ( \Z^{V^{\sprod{n+1}}} ) \right. }
		\bm{e} \left( \frac{1}{4k} {}^t\!\mu W^{\sprod{n+1}} \mu \right)
		\\
		&\prod_{v \in V^{\sprod{n+1}}} F_v^{\sprod{n}} ( \mu_v, t_v^{\sprod{n}} )
		\prod_{i \in \overline{v}^{\sprod{n}}} 
		G_i^{\sprod{n}} ( \mu_v, t_i^{\sprod{n}} )
		\\
		= \,
		&\frac{1}{\abs{\det M^{\sprod{n+1}}}} 
		\sum_{\left. 
			\mu \in \Z^{V^{\sprod{n+1}}} 
			\middle/
			2k M^{\sprod{n+1}} ( \Z^{V^{\sprod{n+1}}} ) \right. }
		\bm{e} \left( \frac{1}{4k} {}^t\!\mu W^{\sprod{n+1}} \mu \right)
		\prod_{v \in V^{\sprod{n+1}}} F_v^{\sprod{n+1}} ( \mu_v, t_v^{\sprod{n+1}} ).
	\end{align}
	Thus, we obtain the claim. 
\end{proof}

\begin{rem}
	The establishment of \cref{lem:pruned_F} is one of motivation to define $ F_v^{\sprod{n}} ( \mu_v, t_v^{\sprod{n}} ) $.
	A strategy to prove this lemma is the same as a proof of \cite[Proposition 6.1]{MM} and \cite[Proposition 3.2]{M_plumbed}.
	We reached the method of ``pruning trees'' by examining these proofs.
\end{rem}

\begin{lem} \label{lem:F_hol}
	For $ n \in \Z_{\ge 0} $, $ v \in V^{\sprod{n}} $ and $ \mu \in \Z $, it holds that
	\[
	\ord_{t_v = 0} F_v^{\sprod{n}} ( \mu, t_v^{\sprod{n}} )
	=
	\begin{cases}
		0 & \text{if } \mu \notin k\Z, \\
		2 - \deg v & \text{if } \mu \in k\Z.
	\end{cases}
	\]
\end{lem}

\begin{proof}
	By induction, we have
	\[
	F_v^{\sprod{n}} ( \mu, t_v^{\sprod{n}} )
	\in
	F_v \left( \zeta_{2k}^{\mu} e^{t_v} \right)
	\bbC(( t_i \mid i \in V_v^{\sprod{n}} \smallsetminus \{ v \} )).
	\]
	Here we have
	\[
	\ord_{t_v = 0} F_v \left( \zeta_{2k}^{\mu} e^{t_v} \right)
	=
	\ord_{t_v = 0} \left( \zeta_{2k}^{\mu} e^{t_v} - \zeta_{2k}^{-\mu} e^{-t_v} \right)^{2 - \deg v}
	=
	\begin{cases}
		0 & \text{if } \mu \notin k\Z, \\
		2 - \deg v & \text{if } \mu \in k\Z.
	\end{cases}
	\]
	Since we have $ F_v^{\sprod{n}} ( \mu, t_v^{\sprod{n}} ) \neq 0 $ and $ G_v^{\sprod{n}} ( \mu, t_v^{\sprod{n}} ) \neq 0 $ by induction, we obtain the desired result.
\end{proof}

To state our final key lemma, we need the following notation.

\begin{nota*}
	For $ N \in \Z_{>0} $ and $ M \in \Z $, denote the $ \bbC[[ t_1, \dots, t_N ]] $-submodule of 
	$ \bbC(( t_1, \dots, t_N )) $ by
	\[
	\bbC(( t_1, \dots, t_N ))_{\ge M}
	\coloneqq
	\left\{
	\sum_{ m \in \Z^N, \, m_1 + \cdots + m_N \ge M }
	a_m t_1^{m_1} \cdots t_N^{m_N}
	\in \bbC(( t_1, \dots, t_N ))
	\right\}.
	\]
	%	For a finite set $ \Lambda $ and $ M \in \Z $, denote the $ \bbC[[ t_\lambda \mid \lambda \in \Lambda ]] $-submodule of 
	%	$ \bbC(( t_\lambda \mid \lambda \in \Lambda )) $ by
	%	\[
	%	\bbC(( t_\lambda \mid \lambda \in \Lambda ))_{\ge M}
	%	\coloneqq
	%	\left\{
	%		\sum_{ m \in \Z^\Lambda, \, \sum_{\lambda \in \Lambda} m_\lambda \ge M }
	%		a_m \prod_{\lambda \in \Lambda} t_\lambda^{m_\lambda}
	%		\in \bbC(( t_\lambda \mid \lambda \in \Lambda ))
	%	\right\}.
	%	\]
\end{nota*}

The following lemma is the final key lemma, which is the most crucial for the proof of the main theorem.

\begin{lem} \label{lem:F_order_sum}
	For $ n \in \Z_{\ge 0} $, $ v \in V^{\sprod{n}} $, $ i \in V^{\sprod{n}} \smallsetminus V^{\sprod{n+1}} $ and $ \mu \in \Z $, it holds that
	\begin{align}
		F_v^{\sprod{n}} ( \mu, t_v^{\sprod{n}} )
		&\in
		\begin{cases}
			\bbC(( t_v^{\sprod{n}} ))_{\ge 0} & \text{if } \mu \notin k\Z, \\
			\bbC(( t_v^{\sprod{n}} ))_{\ge 2 - \deg_{\Gamma^{\sprod{n}}}} & \text{if } \mu \in k\Z,
		\end{cases}
		\\
		G_i^{\sprod{n}} ( \mu, t_i^{\sprod{n}} )
		&\in
		\begin{cases}
			\bbC(( t_i^{\sprod{n}} ))_{\ge 0} & \text{if } \mu \notin k\Z, \\
			\bbC(( t_i^{\sprod{n}} ))_{\ge 1} & \text{if } \mu \in k\Z.
		\end{cases}
	\end{align}
\end{lem}

\begin{proof}
	We prove this by induction.
	For $ F_v^{\sprod{0}} $, the claim holds because 
	\[
	F_v^{\sprod{0}} (\mu, t_v^{\sprod{0}}) = ( \zeta_{2k}^{\mu} e^{t_v} - \zeta_{2k}^{-\mu} e^{-t_v} )^{2 - \deg v}
	\in
	\begin{cases}
		\bbC[[ t_v ]] & \text{if } \mu \notin k\Z, \\
		t_v^{2 - \deg v} \bbC[[ t_v ]] & \text{if } \mu \in k\Z.
	\end{cases}
	\]
	
	Suppose that the claim holds for $ F_v^{\sprod{n}} $.
	First, we prove the claim for $ G_i^{\sprod{n}} $.
	For any $ i \in V^{\sprod{n}} \smallsetminus V^{\sprod{n+1}} $ and $ \mu \in \Z $, we have
	\[
	G_i^{\sprod{n}} ( \mu, t_i^{\sprod{n}} )
	\in
	\sprod{ F_i^{\sprod{n}} ( \mu_i, t_i^{\sprod{n}} ) \relmiddle| \mu_i \in \Z }_{\bbC}
	\subset
	\bbC(( t_i^{\sprod{n}} ))_{\ge 0}
	\]
	by the definition and the induction hypothesis.
	We have $ G_i^{\sprod{n}} ( k \mu, t_i^{\sprod{n}} ) = -G_i^{\sprod{n}} ( k \mu, -t_i^{\sprod{n}} ) $ since
	\begin{align}
		G_i^{\sprod{n}} ( k \mu, t_i^{\sprod{n}} )
		&=
		\bm{e} \left( \frac{k \mu^2}{4 w_i^{\sprod{n}}} \right)
		\sum_{\mu_i \in \Z / 2 M_i^{\sprod{n}} \Z}
		\bm{e} \left( \frac{k}{4} w_i^{\sprod{n}} \mu_i^2 + \frac{1}{2} \mu \mu_i \right)
		F_i^{\sprod{n}} (\mu_i, t_i^{\sprod{n}})
		\\
		&=
		\bm{e} \left( \frac{k \mu^2}{4 w_i^{\sprod{n}}} \right)
		\sum_{\mu_i \in \Z / 2 M_i^{\sprod{n}} \Z}
		\bm{e} \left( \frac{k}{4} w_i^{\sprod{n}} \mu_i^2 + \frac{1}{2} \mu \mu_i \right)
		F_i^{\sprod{n}} (-\mu_i, t_i^{\sprod{n}})
	\end{align}
	and
	\[
	F_i^{\sprod{n}} (-\mu_i, t_i^{\sprod{n}})
	=
	(-1)^{\deg_{\Gamma^{\sprod{n}}} i}
	F_i^{\sprod{n}} (\mu_i, -t_i^{\sprod{n}})
	=
	-F_i^{\sprod{n}} (\mu_i, -t_i^{\sprod{n}})
	\]
	by \cref{lem:F_symmetry}.
	Thus, we have $ G_i^{\sprod{n}} ( k \mu, t_i^{\sprod{n}} ) \in \bbC(( t_i^{\sprod{n}} ))_{\ge 1} $.
	
	Second, we prove the claim for $ F_v^{\sprod{n+1}} $ under the induction hypothesis for $ F_v^{\sprod{n}} $ and $ G_v^{\sprod{n}} $.
	For any $ \mu \in \Z \smallsetminus k\Z $, we have
	\[
	F_v^{\sprod{n+1}} ( \mu, t_v^{\sprod{n+1}} )
	\in
	\left\{
	\prod_{i \in \{ v \} \cup \overline{v}^{\sprod{n}}} f_i (t_i^{\sprod{n}})
	\relmiddle|
	f_i (t_i^{\sprod{n}}) \in \bbC(( t_i^{\sprod{n}} ))_{\ge 0}
	\right\}
	\subset
	\bbC(( t_v^{\sprod{n+1}} ))_{\ge 0}.
	\]
	For any $ v \in V^{\sprod{n+1}} \smallsetminus V^{\sprod{n+2}} $ and $ \mu \in k\Z $, we also have
	\begin{align}
		F_v^{\sprod{n+1}} ( \mu, t_v^{\sprod{n+1}} )
		&\in
		\left\{
		f_v (t_v^{\sprod{n}})
		\prod_{i \in \overline{v}^{\sprod{n}}} g_i (t_i^{\sprod{n}})
		\relmiddle|
		f_v (t_v^{\sprod{n}}) \in \bbC(( t_v^{\sprod{n}} ))_{\ge 2 - \deg_{\Gamma^{\sprod{n}}} v}, \,
		g_i (t_i^{\sprod{n}}) \in \bbC(( t_i^{\sprod{n}} ))_{\ge 1}
		\right\}
		\\
		&\subset
		\bbC(( t_v^{\sprod{n+1}} ))_{\ge 2 - \deg_{\Gamma^{\sprod{n}}} v + \abs{\overline{v}^{\sprod{n}}}}
		=
		\bbC(( t_v^{\sprod{n+1}} ))_{\ge 2 - \deg_{\Gamma^{\sprod{n+1}}} v}
	\end{align}
	by \cref{rem:pruned} \cref{item:rem:pruned:deg}.
\end{proof}

Finally, we give a proof of \cref{thm:phi_properties}.

\begin{proof}[Proof of $ \cref{thm:phi_properties} $]
	\cref{item:thm:phi_properties:non-negative}	
	There exists $ N \in \Z_{\ge 0} $ such that $ 1 \le \abs{V^{\sprod{N}}} \le 2 $. 
	Since $ V^{\sprod{N+1}} = \emptyset $, we have $ V^{\sprod{N}} = V^{\sprod{N}} \smallsetminus V^{\sprod{N+1}} $.
	Thus, by \cref{lem:pruned_F,lem:F_order_sum}, we obtain
	\[
	\varphi_{\Gamma, k} (t_V)
	\in
	\sprod{
		\prod_{v \in V^{\sprod{N}}} F_v^{\sprod{N}} ( \mu_v, t_v^{\sprod{N}} )
		\relmiddle|
		(\mu_v) \in \Z^{V^{\sprod{N}}}
	}_\bbC
	\subset
	\bbC(( t_V ))_{\ge 0}.
	\]
	
	\cref{item:thm:phi_properties:constant_only}
	Take $ m \in \Z^V $ with $ B_m (\Gamma, k) \neq 0 $ and $ \sum_{v \in V} m_v = 0 $.
	Assume the existence of $ v_0 \in V^{\sprod{n}} $ for some $ n $ with $ m_{v_0} \le -1 $ and derive a contradiction.
	We can assume that $ n $ is maximal in such $ n $.
	By \cref{lem:pruned_F}, we have
	\[
	\varphi_{\Gamma, k} (t_V)
	\in
	\sprod{
		\prod_{v \in V^{\sprod{n}}} F_v^{\sprod{n}} ( \mu_v, t_v^{\sprod{n}} )
		\relmiddle|
		(\mu_v) \in \Z^{V^{\sprod{n}}}
	}_\bbC.
	\]
	For any $ v \in V^{\sprod{n}} $ with $ \deg_{\Gamma^{\sprod{n}}} v \ge 3 $ and $ \mu_v \in k \Z $, 
	we have $ \ord_{t_v = 0} F_v^{\sprod{n}} ( \mu_v, t_v^{\sprod{n}} ) = 2 - \deg_{\Gamma^{\sprod{n}}} v \le -1 $ by \cref{lem:F_hol}.
	On the other hand, we have $ m_v \ge 0 $ because $ m_{v'} \ge 0 $ for any $ v' \in V^{\sprod{n+1}} $ by the maximality of $ n $.
	Thus, the Laurent coefficient of $ t_V^m $ of such $ F_v^{\sprod{n}} ( \mu_v, t_v^{\sprod{n}} ) $ is $ 0 $.
	Therefore, such $ F_v^{\sprod{n}} ( \mu_v, t_v^{\sprod{n}} ) $ does not contribute to $ B_m (\Gamma, k) $.
	We also have $ F_v^{\sprod{n}} ( \mu_v, t_v^{\sprod{n}} ) \in \bbC(( t_v^{\sprod{n}} ))_{\ge 0} $ 
	for any $ v \in V^{\sprod{n}} $ with $ \deg_{\Gamma^{\sprod{n}}} v \le 2 $ and $ \mu_v \in \Z $ by \cref{lem:F_order_sum}.
	Since $ v_0 \in V^{\sprod{n}} \smallsetminus V^{\sprod{n+1}} $ by the maximality of $ n $, 
	we have 
	$ F_{v_0}^{\sprod{n}} ( \mu_{v_0}, t_{v_0}^{\sprod{n}} ) \in \bbC(( t_{v_0}^{\sprod{n}} ))_{\ge 0} $ for any $ \mu_{v_0} \in \Z \smallsetminus k \Z $ 
	and
	$ F_{v_0}^{\sprod{n}} ( \mu_{v_0}, t_{v_0}^{\sprod{n}} ) \in \bbC(( t_{v_0}^{\sprod{n}} ))_{\ge 1} $ for any $ \mu_{v_0} \in k \Z $
	by \cref{lem:F_order_sum}.
	Since $ m_{v_0} \ge -1 $, $ F_{v_0}^{\sprod{n}} ( \mu_{v_0}, t_{v_0}^{\sprod{n}} ) $ for $ \mu_{v_0} \in \Z \smallsetminus k \Z $ does not contribute to $ B_m (\Gamma, k) $.
	Thus, $ B_m (\Gamma, k) $ is the coefficient of $ t_V^m $ of an element in
	\begin{align}
		&\sprod{
			\prod_{v \in V^{\sprod{n}}} F_v^{\sprod{n}} ( \mu_v, t_v^{\sprod{n}} )
			\relmiddle|
			\begin{gathered}
				(\mu_v) \in \Z^{V^{\sprod{n}}}, \quad
				\mu_{v_0} \in k \Z, \\
				\mu_v \in \Z \smallsetminus k \Z \text{ for any } v \in V^{\sprod{n}} \text{ with } \deg_{\Gamma^{\sprod{n}}} v \ge 3
			\end{gathered}
		}_\bbC
		\\
		\subset \,
		&\bbC(( t_v \mid v \in V \smallsetminus V_{v_0}^{\sprod{n}} ))_{\ge 0}
		\cdot \bbC(( t_{v_0}^{\sprod{n}} ))_{\ge 1}
		\\
		\subset \,
		&\bbC(( t_V ))_{\ge 1}.
	\end{align}
	Since $ B_m (\Gamma, k) \neq 0 $, we have $ \sum_{v \in V} m_v \ge 1 $.
	This is a contradiction.
	
	\cref{item:thm:phi_properties:constant_expression}
	Take $ N \in \Z_{\ge 0} $ with $ 1 \le \abs{V^{\sprod{N}}} \le 2 $ as in a proof of \cref{item:thm:phi_properties:non-negative}.
	By \cref{lem:pruned_F}, we have
	\[
	\varphi_{\Gamma, k} ( t_V )
	=
	\frac{1}{\abs{\det M^{\sprod{N}}}}
	\sum_{\left. 
		\mu \in \Z^{V^{\sprod{N}}} 
		\middle/
		2k M^{\sprod{N}} ( \Z^{V^{\sprod{N}}} )
		\right. }
	\bm{e} \left( \frac{1}{4k} {}^t\!\mu W^{\sprod{N}} \mu \right)
	\prod_{v \in V^{\sprod{N}}} F_v^{\sprod{N}} ( \mu_v, t_v^{\sprod{N}} ).
	\]
	Since $ V^{\sprod{N}} = V^{\sprod{N}} \smallsetminus V^{\sprod{N+1}} $, we obtain
	$ F_v^{\sprod{N}} ( \mu_v, t_v^{\sprod{N}} ) \in \bbC(( t_v^{\sprod{N}} ))_{\ge 1} $
	for any $ v \in V^{\sprod{N}} $ and $ \mu_v \in k\Z $ by \cref{lem:F_order_sum}.
	Thus, $ B_0 (\Gamma, k) $ is the constant term of
	\[
	\frac{1}{\abs{\det M^{\sprod{N}}}}
	\sum_{\left. 
		\mu \in (\Z \smallsetminus k\Z)^{V^{\sprod{N}}} 
		\middle/
		2k M^{\sprod{N}} ( \Z^{V^{\sprod{N}}} ) 
		\right. }
	\bm{e} \left( \frac{1}{4k} {}^t\!\mu W^{\sprod{N}} \mu \right)
	\prod_{v \in V^{\sprod{N}}} F_v^{\sprod{N}} ( \mu_v, t_v^{\sprod{N}} ).
	\]
	By retracing the calculations in the proof of \cref{lem:pruned_F}, we see that this equation is equal to
	\[
	\frac{1}{\abs{\det M^{\sprod{N-1}}}}
	\sum_{\left. 
		\mu \in 
		\left( (\Z \smallsetminus k\Z)^{V^{\sprod{N}}} \times \Z^{V^{\sprod{N-1}} \smallsetminus V^{\sprod{N}}} \right)
		\middle/
		2k M^{\sprod{N-1}} ( \Z^{V^{\sprod{N-1}}} ) 
		\right. }
	\bm{e} \left( \frac{1}{4k} {}^t\!\mu W^{\sprod{N-1}} \mu \right)
	\prod_{v \in V^{\sprod{N-1}}} F_v^{\sprod{N-1}} ( \mu_v, t_v^{\sprod{N-1}} ).
	\]
	Since $ F_v^{\sprod{N}} ( \mu, t_v^{\sprod{N}} ) \in \bbC(( t_v^{\sprod{N}} ))_{\ge 1} $
	for any $ v \in V^{\sprod{N}} \smallsetminus V^{\sprod{N}} $ and $ \mu \in k\Z $ by \cref{lem:F_order_sum},
	$ B_0 (\Gamma, k) $ is the constant term of
	\[
	\frac{1}{\abs{\det M^{\sprod{N-1}}}}
	\sum_{\left. 
		\mu \in 
		(\Z \smallsetminus k\Z)^{V^{\sprod{N-1}}}
		\middle/
		2k M^{\sprod{N-1}} ( \Z^{V^{\sprod{N-1}}} ) 
		\right. }
	\bm{e} \left( \frac{1}{4k} {}^t\!\mu W^{\sprod{N-1}} \mu \right)
	\prod_{v \in V^{\sprod{N-1}}} F_v^{\sprod{N-1}} ( \mu_v, t_v^{\sprod{N-1}} ).
	\]
	By repeating the above arguments inductively, we obtain that $ B_0 $ is the constant term of
	\[
	\frac{1}{\abs{\det M^{\sprod{n}}}}
	\sum_{\left. 
		\mu \in (\Z \smallsetminus k\Z)^{V^{\sprod{n}}} 
		\middle/
		2k M^{\sprod{n}} ( \Z^{V^{\sprod{n}}} ) 
		\right. }
	\bm{e} \left( \frac{1}{4k} {}^t\!\mu W^{\sprod{n}} \mu \right)
	\prod_{v \in V^{\sprod{n}}} F_v^{\sprod{n}} ( \mu_v, t_v^{\sprod{n}} )
	\]
	for any $ 0 \le n \le N $.
	When $ n=0 $, $ F_v^{\sprod{0}} ( \mu_v, t_v^{\sprod{0}} ) = F_v \left( \zeta_{2k}^{\mu_v} e^{t_V} \right) $ is holomorphic at $ t_v = 0 $
	for any $ v \in V $ and	$ \mu_v \in \Z \smallsetminus k\Z $.
	Thus, we obtain
	\[
	B_0 (\Gamma, k)
	=
	\sum_{\mu \in (\Z \smallsetminus k\Z)^V/2k\Z^V}
	\bm{e} \left( \frac{1}{4k} {}^t\!\mu W \mu \right)
	\prod_{v \in V} F_v \left( \zeta_{2k}^{\mu_v} \right).
	\]
\end{proof}

% --------------------------------------------------------------------------

%\section*{Data availability} \label{sec:data_availability}

% --------------------------------------------------------------------------

%Data sharing is not applicable to this article as no datasets were generated or analyzed during the current study.

% --------------------------------------------------------------------------

%\section*{Ethics declarations} \label{sec:ethics}

% --------------------------------------------------------------------------

%\subsubsection*{Conflict of interest}
%
%The author has no conflicts of interest directly relevant to the content of this article.

% --------------------------------------------------------------------------
%		参考文献
% --------------------------------------------------------------------------

\bibliographystyle{alpha}
\bibliography{GPPV_conj} % myrefs.bibに書いた文献データを引用
% 日本語の書籍タイトルがゴシック体になる. 見苦しいようなら\emphコマンドを書き換える. 

% --------------------------------------------------------------------------
\end{document}